\documentclass[a4paper]{amsart}
\usepackage[utf8]{inputenc}
\usepackage[T1]{fontenc}
\usepackage{lmodern}
\usepackage{amssymb,amsxtra}
\usepackage[all]{xy}
\usepackage{nicefrac,mathtools,enumitem}
\usepackage{mdwlist}
\setlist[enumerate,1]{label=\textup{(\arabic*)}}
\usepackage{lmodern,textcomp}

\usepackage{tikz}
\usetikzlibrary{matrix}
\tikzset{cd/.style=matrix of math nodes,row sep=2em,column sep=2em, text height=1.5ex, text depth=0.5ex}
\tikzset{cdar/.style=->,auto}
\tikzset{overar/.style={draw=white,double=black,double distance=.4pt,very thick}}

\usetikzlibrary{positioning,calc}
\usetikzlibrary{shapes}
\usetikzlibrary{matrix}
\usetikzlibrary{arrows}

\usepackage{microtype}
\usepackage[pdftitle={Braided C*-quantum groups and their bosonizations},
  pdfauthor={Sutanu Roy},
  pdfsubject={Mathematics; MSC }
]{hyperref}

\usepackage[lite]{amsrefs}

\renewcommand{\PrintDOI}[1]{\href{http://dx.doi.org/\detokenize{#1}}{doi: \detokenize{#1}}}

\BibSpec{book}{%
  +{}  {\PrintPrimary}                {transition}
  +{,} { \textit}                     {title}
  +{.} { }                            {part}
  +{:} { \textit}                     {subtitle}
  +{,} { \PrintEdition}               {edition}
  +{}  { \PrintEditorsB}              {editor}
  +{,} { \PrintTranslatorsC}          {translator}
  +{,} { \PrintContributions}         {contribution}
  +{,} { }                            {series}
  +{,} { \voltext}                    {volume}
  +{,} { }                            {publisher}
  +{,} { }                            {organization}
  +{,} { }                            {address}
  +{,} { \PrintDateB}                 {date}
  +{,} { }                            {status}
  +{}  { \parenthesize}               {language}
  +{}  { \PrintTranslation}           {translation}
  +{;} { \PrintReprint}               {reprint}
  +{.} { }                            {note}
  +{.} {}                             {transition}
  +{} { \PrintDOI}                   {doi}
  +{} { available at \url}            {eprint}
  +{}  {\SentenceSpace \PrintReviews} {review}
}

\numberwithin{equation}{section}

\theoremstyle{plain}
\newtheorem{theorem}[equation]{Theorem}
\newtheorem{lemma}[equation]{Lemma}
\newtheorem{proposition}[equation]{Proposition}

\newtheorem{corollary}[equation]{Corollary}

\theoremstyle{definition}
\newtheorem{definition}[equation]{Definition}

\theoremstyle{remark}

\newtheorem{example}[equation]{Example}

\newcommand{\tenscorep}{\mathbin{\begin{tikzpicture}[baseline,x=.75ex,y=.75ex] \draw[line width=.2pt] (-0.8,1.15)--(0.8,1.15);\draw[line width=.2pt](0,-0.25)--(0,1.15); \draw[line width=.2pt] (0,0.75) circle [radius = 1];\end{tikzpicture}}}

\newcommand*{\Braiding}[2]{\begin{tikzpicture}[baseline]
    \draw[-] (0,0) -- (1.4ex,1.4ex) node[right,inner sep=0pt] {$\scriptstyle #2$};
    \draw[-,draw=white,line width=2.4pt] (0,1.4ex) -- (1.4ex,0);
    \draw[-] (1.4ex,0) -- (0,1.4ex) node[left,inner sep=0pt] {$\scriptstyle #1$};
  \end{tikzpicture}}
\newcommand*{\Dualbraiding}[2]{\begin{tikzpicture}[baseline]
    \draw[-] (1.4ex,0) -- (0,1.4ex) node[left,inner sep=0pt] {$\scriptstyle #1$};
    \draw[-,draw=white,line width=2.4pt] (0,0) -- (1.4ex,1.4ex);
    \draw[-] (0,0) -- (1.4ex,1.4ex) node[right,inner sep=0pt] {$\scriptstyle #2$};
  \end{tikzpicture}}

\newcommand*{\corep}[1]{\textup{#1}}          
\newcommand*{\Corep}[1]{\mathbb{#1}}          
\newcommand*{\Ducorep}[1]{\hat{\corep{#1}}}   
\newcommand*{\DuCorep}[1]{\hat{\Corep{#1}}}   

\newcommand*{\Codouble}[1]{\mathfrak{D}({#1})\sphat\text{\space}}

\newcommand*{\Bialg}[1]{(#1,\Comult[#1])}
\newcommand*{\DuBialg}[1]{(\hat{#1},\DuComult[#1])}

\newcommand*{\YDcat}{\mathcal{YD}\mathfrak{C^*alg}}

\usepackage{dsfont}

\newcommand*{\Nhatdrinf}{\hat{\mathcal{N}}}
\newcommand*{\udrinf}{\mathcal U}

\newcommand*{\Mod}[1]{\abs{#1}}
\newcommand*{\Ph}[1]{\Phi_{#1}}

\newcommand*{\nb}{\nobreakdash}
\newcommand*{\Star}{$^*$\nb-}

\newcommand*{\C}{\mathbb C}
\newcommand*{\Z}{\mathbb Z}
\newcommand*{\R}{\mathbb R}

\newcommand*{\T}{\mathbb T}

\newcommand*{\CLS}{\mathrm{CLS}}
\newcommand*{\G}[1][G]{\mathbb #1}
\newcommand*{\DuG}[1][G]{\hat{\mathbb{#1}}}

\newcommand*{\Comult}[1][]{\Delta_{#1}}
\newcommand*{\DuComult}[1][]{\hat{\Delta}_{#1}}
\newcommand*{\Qgrp}[2]{\mathbb{#1}=(#2,\Comult[#2])}
\newcommand*{\DuQgrp}[2]{\widehat{\mathbb{#1}}=(\hat{#2},\DuComult[#2])}

\newcommand*{\Bound}{\mathbb B}
\newcommand*{\Comp}{\mathbb K}

\newcommand*{\Contvin}{\textup C_0}
\newcommand*{\Cont}{\textup C}

\newcommand*{\Mor}{\textup{Mor}}
\newcommand*{\Id}{\textup{id}}

\newcommand*{\Multunit}[1][]{\mathbb W^{#1}}
\newcommand*{\multunit}[1][]{\textup W^{#1}}
\newcommand*{\DuMultunit}[1][]{\widehat{\mathbb W}^{#1}}

\newcommand*{\BrMultunit}{\mathbb F}
\newcommand*{\DuBrMultunit}{\widehat{\mathbb F}}
\newcommand*{\brmultunit}{\textup F}

\newcommand*{\ProjBichar}{\mathbb{P}}
\newcommand*{\projbichar}{\textup P}
\newcommand*{\Duprojbichar}{\hat{\projbichar}}

\newcommand*{\Flip}{\Sigma}
\newcommand*{\flip}{\sigma}
\newcommand*{\Cst}{\textup C^*}
\newcommand*{\Cred}{\textup C^*_\textup r}

\newcommand*{\Cstcat}{\mathfrak{C^*alg}}

\newcommand*{\Hils}[1][H]{\mathcal{#1}}
\newcommand*{\Mult}{\mathcal M}
\newcommand*{\U}{\mathcal U}

\newcommand*{\defeq}{\mathrel{\vcentcolon=}}

\newcommand*{\abs}[1]{\lvert#1\rvert}

\newcommand*{\conj}[1]{\overline{#1}}

\begin{document}
\title[Braided quantum groups and their bosonizations]{Braided quantum groups and their bosonizations in the $\textup{C}^*$-algebraic framework}

\author{Sutanu Roy}
\email{sutanu@niser.ac.in}
\address{School of Mathematical Sciences\\
 National Institute of Science Education and Research  Bhubaneswar\\
 Jatni, 752050\\
 India}
\address{Homi Bhabha National Institute\\ 
Training School Complex, Anushaktinagar\\ 
Mumbai, 400094\\ 
India}
 
\begin{abstract}
 We present a general theory of braided quantum groups in the \(\Cst\)\nb-algebraic framework using the language of multiplicative unitaries. Starting with a manageable multiplicative unitary in the representation category of the quantum codouble of a regular quantum group \(\G\) we construct a braided \(\Cst\)\nb-quantum group over~\(\G\) as a \(\Cst\)\nb-bialgebra in the monoidal category of the \(\G\)\nb-Yetter-Drinfeld \(\Cst\)\nb-algebras. Furthermore, we establish the one to one correspondence between braided \(\Cst\)\nb-quantum groups and \(\Cst\)\nb-quantum groups with a projection. Consequently, we generalise the bosonization construction for braided Hopf\nb-algebras of Radford and Majid to braided \(\Cst\)\nb-quantum groups. Several examples are discussed. In particular, we show that the complex quantum plane admits a braided \(\Cst\)\nb-quantum group structure over the circle group~\(\T\) and identify its bosonization with the simplified quantum \(\textup{E}(2)\) group.
\end{abstract}

\maketitle

\section{Introduction}
 \label{sec:intro}
Semidirect product construction of groups is a fundamental method of extending certain homogeneous symmetries to some inhomogeneous symmetries of a given physical system. In the realm of noncommutative geometry, this would mean that the semidirect product of quantum groups may contain information about the inhomogeneous quantum symmetries of quantum spaces. Several investigations were done by several authors to understand the structure of the inhomogeneous quantum groups mostly at the algebraic level. We refer to~\cite{PV1997} and the references therein for more details. The primary focus of those works was on the understanding of the deformations of the Poincar\'e group. 

On the other hand, many important examples of \(\Cst\)\nb-quantum groups~\(\Qgrp{H}{C}\) were constructed by deforming semidirect product of Lie groups~\(K\rtimes G\) with Abelian~\(G\) \cites{W1991a,W2001,WZ2002}. The \(\Cst\)\nb-quantum group structure is captured by a single unitary operator \(\Multunit[C]\in\U(\Hils\otimes\Hils)\), on a suitably chosen separable Hilbert space~\(\Hils\), with two additional properties. The first one is algebraic namely, \(\Multunit[C]\) is a \emph{multiplicative unitary}: \(\Multunit[C]\) satisfies the 
 \emph{pentagon equation}
  \begin{equation}
    \label{eq:pentagon}
    \Multunit[C]_{23}\Multunit[C]_{12}
    = \Multunit[C]_{12}\Multunit[C]_{13}\Multunit[C]_{23}
    \qquad
    \text{in \(\U(\Hils\otimes\Hils\otimes\Hils)\).}
  \end{equation}  
 This ensures the set~\(C_{0}\defeq\{(\omega\otimes\Id_{\Hils})\Multunit[C]\mid \omega\in\Bound(\Hils)_{*}\}\subset\Bound(\Hils)\) is an algebra. The analytic property of \(\Multunit\), namely \emph{manageability}~\cite{W1996}*{Definition 1.1}, implies 
  \begin{equation}
 \label{eq:5Aug20}
   C=C_{0}^\CLS\text{ is a \(\Cst\)\nb-algebra and }
   \Comult[C](c)\defeq \Multunit[C](c\otimes 1)\Multunit[C]{}^{*},
   \quad
   \text{for all~\(c\in C\).}    
\end{equation}
 Here, \(\CLS\) stands for the closed linear span. 
The \(\Cst\)\nb-quantum group~\(\Qgrp{H}{C}\) is then said to be 
 \emph{generated} by~\(\Multunit[C]\) in the sense of~\cite{W1996}*{Theorem 1.5}. 

The semidirect product group \(K\rtimes G\) comes with a canonical 
endomorphism~\(p\colon K\rtimes G\to K\rtimes G\) defined by~\(p(k,g)=(1_{K},g)\), where~\(1_{K}\) denotes the identity element of~\(K\). Clearly, \(p\) is idempotent, that is~\(p^{2}=p\), with the image~\(G\subset K\rtimes G\) and the kernel~\(K\subset K\rtimes G\). This fact gets translated to the deformations of~\(K\rtimes G\) as well. At the level of multiplicative unitaries, there exist unitaries \(\BrMultunit ,\ProjBichar\in\U(\Hils\otimes\Hils)\) such that~\(\ProjBichar\) is again a manageable multiplicative unitary and~\(\Multunit=\BrMultunit\ProjBichar\). Also, \(\ProjBichar\) generates~\(G\), while viewed as quantum group \(\Qgrp{G}{\Contvin(G)}\), as a Woronowicz closed quantum subgroup of~\(\G[H]\), see~\cite{DKSS2012}*{Definition 3.2}. 
Thus \(\ProjBichar\) is the quantum analogue of the idempotent group homomorphism or \emph{projection} on~\(\G[H]\). At the \(\Cst\)\nb-algebra level, \(C\) is identified with the crossed product \(\Cst\)\nb-algebra \(B\rtimes\widehat{G}\) for some \(\Cst\)\nb-algebra \(B\) equipped with an action of the dual group~\(\widehat{G}\).

However, the range of the restriction of~\(\Comult[C]\) on~\(B\) is not a \(\Cst\)\nb-subalgebra of the multiplier algebra of~\(B\otimes B\) denoted by~\(\Mult(B\otimes B)\) and \(\BrMultunit\) is not a multiplicative unitary. 
This strongly indicates that the quantum analogue of the translation group~\(K\) is not a \(\Cst\)\nb-quantum groups.

In a purely algebraic setting, when quantum groups and Hopf algebras are synonymous, Radford had discovered~\cite{R1985} that the Hopf algebras \(C\) with a projection~\(p\) is equivalent to pairs consisting of a Hopf algebra \(A=\textup{Im}(p)\) and a braided Hopf algebra \(B\) over \(A\). We refer~\cite{M1995}*{Chapter 10} for a detailed discussion on it. This was further generalised in the categorical framework and extensively studied by Majid~\cites{M1994, M1999, M2000}. The reconstruction of the Hopf algebra \(C\) and the projection~\(p\) starting from \(A\) and~\(B\) is named by Majid as \emph{bosonization}.

Motivated by the algebraic theory we ask the following question: does there exist a one to one correspondence between \emph{braided \(\Cst\)\nb-quantum groups} and quantum groups with projection? A systematic investigation in this direction was initiated by the author in his thesis~\cite{R2013}. It was further studied in~\cites{MRW2017,MR2019} at the level of manageable multiplicative unitaries, and in~\cite{KS2015} at the level of von Neumann algebras. Meanwhile the \emph{braided compact quantum groups} over a compact quantum group~\(\G\) was introduced in~\cite{MRW2016}. The \(\Cst\)\nb-algebra version of the associated bosonization turns out to be a compact quantum group. Consequently, \(q\)\nb-deformations of \(\textup{SU}(2)\) group, braided analogue of the free orthogonal groups \(\textup{O}(n)\) (in dimension~\(n\)) for nonzero~\(q\in\C\) with~\(\textup{Phase}(q)\neq 1\) were constructed as braided compact quantum over~\(\T\), see~\cites{KMRW2016, MR2019a}. In fact, the resulting bosonizations of braided~\(\textup{SU}_{q}(2)\) groups are~\(\textup{U}_{q}(2)\) groups. Furthermore, the braided compact quantum groups constructed in the recent works~\cites{R2021,BJR2022} captures quantum symmetries of matrix algebras and graph~\(\Cst\)\nb-algebras.

\medskip

The goal of this article is twofold. First, we provide an avenue to pass from manageable braided multiplicative unitaries to braided \(\Cst\)\nb-quantum groups in Theorem~\ref{the:braid_Qgrp}. Consequently, we construct the duals of braided \(\Cst\)\nb-quantum groups as braided \(\Cst\)\nb-quantum groups and generalise Pontrjagin duality for braided \(\Cst\)\nb-quantum groups. Secondly, we establish the one to one correspondence between braided \(\Cst\)\nb-quantum groups and \(\Cst\)\nb-quantum groups with projection in Theorem~\ref{the:brdq_qgpp}. In particular, this allows to construct new examples of \(\Cst\)\nb-quantum groups using braided~\(\Cst\)\nb-quantum groups over~\(\G\) as building blocks. The resulting theory turns out to be very general and it covers the following:
\begin{enumerate}
 \item quantum \(\textup{E}(2)\) groups associated to nonzero real deformation parameters~\cite{W1991a}, quantum~$az+b$ groups~\cites{W2001, S2005} and quantum $ax+b$ groups~\cites{WZ2002} are bosonizations of some braided \(\Cst\)\nb-quantum groups;
\item braided compact quantum groups over a compact quantum group~\cite{MRW2016} are braided \(\Cst\)\nb-quantum groups and so are their examples~\cites{KMRW2016, MR2019a,R2021,BJR2022};
 \item \(q\)\nb-deformations of \(\textup{E}(2)\) group are braided \(\Cst\)\nb-quantum groups over~\(\T\) for~\(q\in\{z\in\C\mid 0<|z|<1\}\setminus\R\), and their bosonizations provide new examples of \(\Cst\)\nb-quantum groups~\cite{RR2020};
 \item the \emph{complex quantum planes} associated to the real deformation parameters~\(0<q<1\) are braided \(\Cst\)\nb-quantum groups over~\(\T\) with Woronowicz's simplified \(\textup{E}_{q}(2)\) groups as their bosonization, see Section~\ref{sec:qnt_E_2}.
\end{enumerate}

Let us briefly describe the techniques we have employed to develop this theory and give an outline of the article. We begin by fixing notations, recalling the necessary definitions, and results in Section~\ref{sec:prelim}.

Suppose~\(\mathcal{C}\) is the category of unitary representations of the quantum codouble of a \(\Cst\)\nb-quantum group~\(\G\) on separable Hilbert spaces. Then \(\mathcal{C}\) is a braided monoidal category and the braiding operators are unitaries, see \cite{MRW2016}*{Proposition 3.4 \& Section 5}. In short, we call~\(\mathcal{C}\) as unitarily braided monoidal category. A unitary morphism \(\BrMultunit\colon \Hils[L]\otimes\Hils[L]\to 
 \Hils[L]\otimes\Hils[L]\) in~\(\mathcal{C}\) is a 
 \emph{braided multiplicative unitary over \(\G\)} if it  
 satisfies a variant of the pentagon equation~\eqref{eq:pentagon} in~\(\mathcal{C}\):
  \begin{equation}
      \label{eq:top-braided_pentagon}
      \BrMultunit_{23} \BrMultunit_{12}
      = \BrMultunit_{12} (\Braiding{\Hils[L]}{\Hils[L]})_{23}
      \BrMultunit_{12} (\Dualbraiding{\Hils[L]}{\Hils[L]})_{23}
      \BrMultunit_{23}
      \quad\text{in }\U(\Hils[L]\otimes\Hils[L]\otimes\Hils[L]),
  \end{equation}  
 where~\(\Braiding{\Hils[L]}{\Hils[L]}\in\U(\Hils[L]\otimes\Hils[L])\) is the unitary braiding and \(\Dualbraiding{\Hils[L]}{\Hils[L]}\defeq (\Braiding{\Hils[L]}{\Hils[L]})^{*}\). Next we 
 assume~\(\BrMultunit\) is manageable~\cite{MRW2017}*{Definition 3.5} and define
 \begin{equation}
  \label{eq:Aug6}
   B_{0}=\{(\omega\otimes\Id_{\Hils[L]}) \BrMultunit\mid \omega\in\Bound(\Hils[L])_{*}\}, 
   \quad 
  B=B_{0}^{\CLS},
  \quad 
  \Comult[B](b)\defeq \BrMultunit (b\otimes 1)\BrMultunit^*. 
 \end{equation}
Unlike unbraided situation, it is unclear whether \(B_{0}\) is an algebra in the first place. In order to prove that~\(B\) is \(\Cst\)\nb-algebra, we consider the \(\Cst\)\nb-quantum group~\(\Qgrp{H}{C}\), generated by the manageable multiplicative unitary~\(\Multunit[C]\), with a projection~\(\ProjBichar\) associated to~\(\BrMultunit\) given by~\cite{MRW2017}*{Theorem 3.7}. Here we use the manageability of~\(\BrMultunit\) implicitly. Next we ensure that~\(\ProjBichar\) generates~\(\G\) and it is a Woronowicz closed quantum subgroup of~\(\G[H]\) in Proposition~\ref{prop:G-prod_proj}. The ``kernel''~\(\ProjBichar\) corresponds to the quantum homogeneous space~\(\G \backslash\G[H]\) with respect to the (left) quantum group homomorphism~\(\Delta_{L}\colon C\to\Mult(A\otimes C)\) that corresponds to~\(\ProjBichar\), see~\cite{MRW2012}*{Theorem 5.5}. At this point, we assume~\(\G\) is a regular quantum group~\cite{BS1993}. Then the existence and uniqueness, up to \(\G\)\nb-equivariant isomorphism, of the underlying \(\Cst\)\nb-algebra of the quantum homogeneous space~\(\G\backslash\G[H]\) inside~\(\Mult(C)\) follows from the Landstad-Vaes theory~\cites{V2005, RW2018}. In fact, Proposition~\ref{prop:Landstad_slices} is an important step where we show that the underlying \(\Cst\)\nb-algebra of~\(\G\backslash\G[H]\subset\Mult(C)\) is unitarily equivalent to \(B\); hence~\(B\) is a \(\Cst\)\nb-algebra.

Consequently, we prove the first main result of this article Theorem~\ref{the:braid_Qgrp} namely, the construction of the \emph{braided \(\Cst\)\nb-quantum group}~\(\Bialg{B}\) over~\(\G\) from~\(\BrMultunit\). More precisely, 
\(B\subset\Bound(\Hils[L])\) is a \(\G\)\nb-Yetter\nb-Drinfeld \(\Cst\)\nb-algebra and 
\(\Comult[B]\colon B\to\Mult(B\boxtimes B)\) is a nondegenerate \Star{}homomorphism satisfying braided analogue of coassociativity and cancellation conditions. Here~\(\boxtimes\) denotes the monoidal product of the category of \(\G\)\nb-Yetter\nb-Drinfeld \(\Cst\)\nb-algebras.

Next, we discuss the bosonization construction for the braided \(\Cst\)\nb-quantum group~\(\Bialg{B}\) over~\(\Qgrp{G}{A}\) by reconstructing~\(\Qgrp{H}{C}\) and the projection~\(\ProjBichar\) in Proposition~\ref{prop:big_qntgrp}. In particular, \(\G\)\nb-Yetter\nb-Drinfeld structure on~\(B\) says that there is an action~\(\hat{\beta}\) of~\(\DuG\) on~\(B\). We identify \(C\) with crossed product \(\Cst\)\nb-algebra~\(B\rtimes_{\hat{\beta}} \DuG\) and express \(\Comult[C]\) in terms of \(\Comult[A]\) and~\(\Comult[B]\). Then we establish the desired one to one correspondence between braided \(\Cst\)\nb-quantum groups and quantum groups with a projection up to isomorphism in Theorem~\ref{the:brdq_qgpp}.

Finally, in Section~\ref{sec:Examples} we show that our theory applies to a large class of examples of \(\Cst\)\nb-quantum groups. In particular, we apply our main results to the concretely constructed example of a manageable braided multiplicative unitary in~\cite{MRW2017}*{Section 4} over~\(\T\). We obtain complex quantum plane as the resulting braided \(\Cst\)\nb-quantum group over~\(\T\) and the simplified quantum \(\textup{E}(2)\) group coincides with the associated bosonization.

\subsection{Acknowledgement} 
This work was partially supported by INSPIRE faculty award Grant~No. DST/INSPIRE/04/2016/000215 given by the D.S.T., Government of India. The author would like to express his sincere gratitude to Professor Ralf Meyer and Professor Stanis\l{}aw Lech Woronowicz for their inspiration. The author is also grateful to Professor Woronowicz for sharing Proposition~\ref{Prop:tens_aff}. The author also thanks the anonymous referees for their constructive inputs. 

\section{Preliminaries}
\label{sec:prelim}
All Hilbert spaces and \(\Cst\)\nb-algebras (which are not explicitly multiplier algebras) are assumed to be separable. For a \(\Cst\)\nb-algebra~\(A\), let \(\Mult(A)\) be its multiplier algebra and let \(\U(A)\) be the group of unitary multipliers of~\(A\) and denote the identity element of~\(\U(A)\) by~\(1_{A}\). 
 For two norm closed subsets \(X\) and~\(Y\) of a 
\(\Cst\)\nb-algebra~\(A\) and~\(T\in\Mult(A)\), we set 
\[
  XY\defeq\{xy\mid x\in X, y\in Y\}^\CLS ,
  \quad
  XTY\defeq\{xTy\mid x\in X, y\in Y\}^\CLS ,
\]
where CLS stands for the \emph{closed linear span}.

Let~\(\Cstcat\) be the category of \(\Cst\)\nb-algebras with
nondegenerate \Star{}homomorphisms \(\varphi\colon A\to\Mult(B)\) as
morphisms \(A\to B\); let \(\Mor(A,B)\) denote the set of morphisms.

Let~\(\Hils\) be a Hilbert space.  A \emph{representation} of a
\(\Cst\)\nb-algebra~\(A\) is a nondegenerate
\Star{}homomorphism \(\pi\colon A\to\Bound(\Hils)\).  Since
\(\Bound(\Hils)=\Mult(\Comp(\Hils))\) and the nondegeneracy
conditions \(\pi(A) \Comp(\Hils)=\Comp(\Hils)\) and
\(\pi(A) \Hils=\Hils\) are equivalent; hence \(\pi\in\Mor(A,\Comp(\Hils))\). 
The unit element of~\(\Mult(\Comp(\Hils))\) is denoted by~\(1_{\Hils}\).

We write~\(\Flip\) for the tensor flip \(\Hils\otimes\Hils[K]\to
\Hils[K]\otimes\Hils\), \(x\otimes y\mapsto y\otimes x\), where
\(\Hils\) and~\(\Hils[K]\) are Hilbert spaces. We write~\(\flip\) for the
tensor flip isomorphism \(A\otimes B\to B\otimes A\) for two
\(\Cst\)\nb-algebras \(A\) and~\(B\). Further we use the standard `leg numbering' 
notation for maps acting on tensor products.

Let~\(\Hils\) be a Hilbert space and let \(D\) be a nondegenerate \(\Cst\)\nb-subalgebra of~\(\Bound(\Hils)\). A closed and densely defined operator~\(T\) acting on~\(\Hils\) is said to be \emph{affiliated with} \(D\) if \(z_{T}\defeq T(I+T^{*}T)^{-\frac{1}{2}}\in\Mult(D)\) and~\((1-z_{T}^{*}z_{T})D\) is dense in \(D\) (see~\cite{W1995}). It is denoted by~\(T\eta D\).

\subsection{C*-quantum groups, their actions and representations}
\label{sec:multunit_quantum_groups} 
A \emph{\(\Cst\)\nb-quantum group}~\(\G\) is a 
pair~\(\Bialg{A}\) consisting of a \(\Cst\)\nb-algebra 
\(A\) and an element~\(\Comult[A]\in\Mor(A,A\otimes A)\) 
\emph{generated} by a manageable multiplicative unitary \(\Multunit\) 
in the way described in the following theorem.
\begin{theorem}[\cites{SW2007,
    W1996}]
  \label{the:Cst_quantum_grp_and_mult_unit}
  Let~\(\Hils\) be a Hilbert space and let
  \(\Multunit\in\U(\Hils\otimes\Hils)\) be a manageable
  multiplicative unitary. Then 
   \begin{enumerate}
   \item the sets of left and right slices of~\(\Multunit\), defined by
  \begin{equation}
    \label{eq:first_leg_slice}
    A  \defeq \{(\omega\otimes\Id_{\Hils})\Multunit\mid
    \omega\in\Bound(\Hils)_*\}^\CLS,
    \quad
    \hat{A} \defeq \{(\Id_{\Hils}\otimes\omega)\Multunit\mid
    \omega\in\Bound(\Hils)_*\}^\CLS,
  \end{equation}
  are nondegenerate \(\Cst\)\nb-subalgebras of~\(\Bound(\Hils)\);
  \item \(\Multunit\in\U(\hat{A}\otimes
    A)\subseteq\U(\Hils\otimes\Hils)\).  We write~\(\multunit\)
    for~\(\Multunit\) viewed as a unitary multiplier of
    \(\hat{A}\otimes A\) and call it \emph{reduced bicharacter}.
 \item \(\Comult[A]\) is uniquely characterised by 
    \begin{equation}
      \label{eq:Comult_W}
      (\Id_{\hat{A}}\otimes \Comult[A])\multunit
      = \multunit_{12}\multunit_{13}
      \qquad \text{in \(\U(\hat{A}\otimes A\otimes A)\).}
    \end{equation}  
    Moreover, 
    \(\Comult[A]\) is \emph{coassociative}\textup{:}
    \begin{equation}
      \label{eq:coassociative}
      (\Comult[A]\otimes\Id_A)\circ\Comult[A]
      = (\Id_A\otimes\Comult[A])\circ\Comult[A],
    \end{equation}
    and satisfies the \emph{cancellation conditions}\textup{:}
    \begin{equation}
      \label{eq:cancel}
      \Comult[A](A) (1_A\otimes A)
      = A\otimes A
      = (A\otimes 1_A) \Comult[A](A).
    \end{equation}
  \end{enumerate}
\end{theorem}
We shall not use the full power of the Haar weight 
approach towards \(\Cst\)\nb-quantum groups developed by Kustermans and 
Vaes in~\cite{KV2000}. The \emph{dual} multiplicative unitary of~\(\Multunit\) is \(\DuMultunit\defeq
\Flip\Multunit[*]\Flip\in\U(\Hils\otimes\Hils)\).  If~\(\Multunit\) is manageable
so is~\(\DuMultunit\).  The dual \(\Cst\)\nb-quantum group
\(\DuQgrp{G}{A}\) generated by~\(\DuMultunit\). Its comultiplication map~\(\DuComult[A]\in\Mor(\hat{A},\hat{A}\otimes\hat{A})\) is uniquely 
determined by the following equation
\begin{equation}
  \label{eq:dual_Comult_W}
  (\DuComult[A]\otimes \Id_A)\multunit
  = \multunit_{23}\multunit_{13}
  \qquad\text{in \(\U(\hat{A}\otimes \hat{A}\otimes A)\).}
\end{equation}
A~\(\Cst\)\nb-quantum group~\(\G\) is \emph{regular} if 
\begin{equation}
  \label{eq:regularity_1}
  (1_{\hat{A}}\otimes A)\multunit(\hat{A}\otimes 1_A)
  = \hat{A}\otimes A,
\end{equation}
see~\cite{BS1993}*{Proposition 3.2 (b) \& Proposition 3.6} and~\cite{SW2007}*{Lemma 40}. 

\begin{example}
 \label{eq:Grp-Qgrp}
 Suppose~\(G\) is a locally compact group. Let~\(\Hils\) be the Hilbert space of square integrable functions with respect to the right Haar measure of~\(G\). We denote the right regular representation of~\(G\) on~\(\Hils\) by~\(\mu\). Define 
 \((\Multunit\xi)(g_{1},g_{2})\defeq\xi(g_{1}g_{2},g_{2})\) for all~\(\xi\in\Hils\otimes\Hils\) and~\(g_{1}, g_{2}\in G\). Then~\(\Multunit\) is a manageable multiplicative unitary and generates the \(\Cst\)\nb-quantum group~\(\Qgrp{G}{\Contvin(G)}\), where 
 \((\Comult[\Contvin(G)]f)(g_{1},g_{2})\defeq f(g_{1}g_{2})\) for all~\(f\in\Contvin(G)\). Also~\(\Bialg{\Cred(G)}\), where~\(\Comult[\Cred(G)](\mu_{g})\defeq \mu_{g}\otimes \mu_{g}\) for all~\(g\in G\), is the dual of~\(\DuG\). In fact, \(\G\) and~\(\DuG\) are examples of regular \(\Cst\)\nb-quantum groups.
\end{example}
 
\begin{definition}
  \label{def:cont_action}
  A \emph{right action} of~\(\G\) on a
  \(\Cst\)\nb-algebra~\(C\) is an injective element~\(\gamma\in\Mor(C, C\otimes
  A)\) with the following properties:
  \begin{enumerate}
    \item \(\gamma\) is a comodule structure, that is,
    \begin{equation}
      \label{eq:right_action}
      (\Id_C\otimes\Comult[A])\circ\gamma
      = (\gamma\otimes\Id_A)\circ\gamma;
    \end{equation}
  \item \(\gamma\) satisfies the \emph{Podle\'s condition}: 
  \(\gamma(C)(1_C\otimes A)=C\otimes A\).
 \end{enumerate}
  We call \((C,\gamma)\) a \emph{\(\G\)\nb-\(\Cst\)\nb-algebra}.  We shall
  drop~\(\gamma\) from our notation whenever it is clear from the context.
\end{definition}
  Similarly, a \emph{left action} of~\(\G\) on~\(C\) is an injective element
  \(\gamma\in\Mor(C, A\otimes C)\) satisfying an appropriate variant 
  of~\eqref{eq:right_action}, that is 
  \((\Comult[A]\otimes\Id_{C})\circ\gamma=(\Id_{A}\otimes\gamma)\circ\gamma\), 
  and the Podle\'s condition: \(\gamma_{C}(C)(A\otimes 1_{C})=A\otimes C\). 
  The word ``action'' will always 
  mean right action throughout.  
  
  For any two \(\G\)\nb-\(\Cst\)\nb-algebras \((C_{1},\gamma_{1})\) and 
  \((C_{2},\gamma_{2})\) an element \(f\in\Mor(C_{1}, C_{2})\) is
  said to be \emph{\(\G\)\nb-equivariant} if \(\gamma_{2} \circ f =
  (f\otimes\Id_A)\circ\gamma_{1}\). The set of
  \(\G\)\nb-equivariant morphisms from~\(C_{1}\) to~\(C_{2}\) is denoted by 
  \(\Mor^{\G}(C_{1},C_{2})\). Let \(\Cstcat(\G)\) be the category 
  with \(\G\)\nb-\(\Cst\)-algebras as
  objects and \(\G\)\nb-equivariant morphisms as arrows.

\begin{definition}
  \label{def:corepresentation}
  A (right) \emph{representation} of~\(\G\) on a Hilbert
  space~\(\Hils[L]\) is a unitary \(\corep{U}\in\U(\Comp(\Hils[L])\otimes A)\)
  with
  \begin{equation}
    \label{eq:corep_cond}
    (\Id_{\Hils[L]}\otimes\Comult[A])\corep{U} =\corep{U}_{12}\corep{U}_{13}
    \qquad\text{in }\U(\Comp(\Hils[L])\otimes A\otimes A).
  \end{equation}
\end{definition}
The \emph{tensor product of representations}~\(\corep{U}^{i}\in\U(\Comp(\Hils[L]_{i})\otimes A)\) of~\(\G\) on~\(\Hils[L]_{i}\) for~\(i=1,2\) 
is defined by~\(\corep{U}^{1}\tenscorep \corep{U}^{2}\defeq\corep{U}^{1}_{13}\corep{U}^{2}_{23}\in\U(\Comp(\Hils[L]_1\otimes\Hils[L]_2)\otimes A)\). 
\begin{definition}
  \label{def:covariant_corep}
  A \emph{covariant representation} of a \(\G\)\nb-\(\Cst\)\nb-algebra \((C,\gamma)\) on 
  a Hilbert space~\(\Hils[L]\) is a pair~\((\corep{U},\varphi)\) consisting of a 
  representation \(\corep{U}\in\U(\Comp(\Hils[L])\otimes A)\) of~\(\G\) and an 
  element~\(\varphi\in\Mor(C,\Comp(\Hils[L]))\) that satisfy the covariance
  condition
  \begin{equation}
    \label{eq:covariant_corep}
    (\varphi\otimes\Id_A)(\gamma(c)) =
    \corep{U}(\varphi(c)\otimes 1_A)\corep{U}^*
    \qquad\text{in }\U(\Comp(\Hils[L])\otimes A)
  \end{equation}
  for all \(c\in C\).  Moreover, \((\corep{U},\varphi)\) is called
  \emph{faithful} if~\(\varphi\) is faithful. Existence of faithful covariant 
  representations is guaranteed by~\cite{MRW2014}*{Example~4.5}.
\end{definition}

\subsection{Heisenberg pairs}
 \label{subsec:Heis_pair}
  Let~\(\Qgrp{G}{A}\) be a~\(\Cst\)\nb-quantum group, let 
  \(\DuQgrp{G}{A}\) be its dual, and 
  let~\(\multunit\in\U(\hat{A}\otimes A)\) 
  be the reduced bicharacter of~\(\G\).
  
  For a pair of representations~\((\pi,\hat{\pi})\) of \(A\) and \(\hat{A}\) 
  on a Hilbert space~\(\Hils\) we 
  denote~\(\multunit_{1\pi}\defeq((\Id_{\hat{A}}\otimes\pi)\multunit)_{12}\) and 
  \(\multunit_{\hat{\pi}3}\defeq((\hat{\pi}\otimes\Id_{A})\multunit)_{23}\) 
  in~\(\U(\hat{A}\otimes\Comp(\Hils)\otimes A)\).
  
  The pair~\((\pi,\hat{\pi})\) is called 
  \begin{enumerate}
   \item a \(\G\)\nb-\emph{Heisenberg pair} if and only if~\(
   \multunit_{\hat{\pi}3}\multunit_{1\pi}
   = \multunit_{1\pi}\multunit_{13}\multunit_{\hat{\pi}3}\); 
  \item a \(\G\)\nb-\emph{anti-Heisenberg pair} 
  if and only if \(\multunit_{1\pi}\multunit_{\hat{\pi}3}
  = \multunit_{\hat{\pi}3}\multunit_{13}\multunit_{1\pi}\).
  \end{enumerate}
  In fact, \(\G\)\nb-Heisenberg pairs and \(\G\)\nb-anti\nb-Heisenberg pairs are 
  in one to one correspondence, see~\cite{MRW2014}*{Lemma 3.4}. 
  A \(\G\)\nb-Heisenberg or a \(\G\)\nb-anti-Heisenberg pair~\((\pi,\hat{\pi})\) 
  is said to be faithful if~\(\pi\) and \(\hat{\pi}\) are faithful representations. \cite{R2015a}*{Proposition 3.2} 
  shows that any~\(\G\)\nb-Heisenberg pair is faithful. 
  
 For any~\(\G\)\nb-Heisenberg or~\(\G\)\nb-anti\nb-Heisenberg pair~\((\pi,\hat{\pi})\) on~\(\Hils\), We denote by~\(\hat{A}_{*}\) the set of all linear functionals on~\(\hat{A}\) that admit extensions to normal functionals on the weak closure of~\(\hat{\pi}(\hat{A})\). It turns out that \(\hat{A}_{*}\) is independent of the choice of~\((\pi,\hat{\pi})\).
  
Consider a pair of representations~\((\corep{U},\corep{V})\) of~\(\G\) and~\(\DuG\) on the Hilbert spaces~\(\Hils[L]_{1}\) and~\(\Hils[L]_{2}\), respectively. By virtue of \cite{MRW2014}*{Theorem 4.1}, for any \(\G\)\nb-Heisenberg pair 
 \((\pi,\hat{\pi})\) on~\(\Hils\) there exists a 
 unique~\(Z\in\U(\Hils[L]_{1}\otimes\Hils[L]_{2})\) such that 
\begin{equation}
 \label{eq:braid}
\corep{U}_{1\pi}\corep{V}_{2\hat{\pi}}Z_{12}
=\corep{V}_{2\hat{\pi}}\corep{U}_{1\pi}
\quad\text{in~\(\U(\Hils[L]_{1}\otimes\Hils[L]_{2}\otimes\Hils)\),}
\end{equation}
where~\(\corep{U}_{1\pi}\defeq 
((\Id_{\Hils[L]_{1}}\otimes\pi)\corep{U})_{13}\) and 
\(\corep{V}_{2\hat{\pi}}\defeq 
((\Id_{\Hils[L]_{2}}\otimes\hat{\pi})\corep{V})_{23}\).

\subsection{Landstad-Vaes theory} 
  Let~\(\Qgrp{G}{A}\) be a~\(\Cst\)\nb-quantum group, let 
  \(\DuQgrp{G}{A}\) be its dual, and 
  let~\(\multunit\in\U(\hat{A}\otimes A)\) 
  be the reduced bicharacter of~\(\G\).
  
 A \emph{\(\G\)\nb-product} is a triple 
 \((C,\gamma, i)\) consisting of a 
 \(\Cst\)\nb-algebra \(C\), a left action 
 \(\gamma\in\Mor(C,A\otimes C)\) of 
 \(\G\) on~\(C\), and an element 
 \(i\in\Mor(A,C)\) satisfying 
  \begin{equation}
      \label{eq:landstad} 
    \gamma \circ i = (\Id_{A}\otimes i)\circ\Comult[A].   
  \end{equation}  
Define \(X\defeq (\Id_{\hat{A}}\otimes i)\multunit \in
\U(\hat{A}\otimes C)\). 
\begin{theorem}
  \label{the:landstad}
    Suppose \(\Qgrp{G}{A}\) is a regular \(\Cst\)\nb-quantum group and 
    \((\pi,\hat{\pi})\) is a \(\G\)\nb-Heisenberg on a Hilbert space~\(\Hils_{\pi}\). 
  Let \((C,\gamma,i)\) be a \(\G\)\nb-product. Define \(\varphi\colon C\to \Comp(\Hils_{\pi})\otimes C\) by 
  \(\varphi(c)\defeq X^*_{\hat{\pi}2}\gamma(c)_{\pi 2}X_{\hat{\pi}2}\) for \(c\in C\). There is a unique
  \(\Cst\)\nb-subalgebra ~\(D\) of~\(\Mult(C)\) with the
  following properties\textup{:}
  \begin{enumerate}
  \item \(D\subseteq \{c\in\Mult(C) \mid\text{ \(\gamma(c)=1_{A}\otimes c\)}\}\)\textup{;}
  \item \(C=i(A) D\)\textup{;}
  \item \(\hat{A}\otimes D = (\hat{A}\otimes 1) \varphi(D)\). 
  \end{enumerate}
  More explicitly, 
  \begin{equation}
  \label{eq:explicit-landstad-vaes}
   D=\{(\omega\otimes\Id_{C})\varphi(c)\mid \omega\in\Bound(\Hils)_{*}, c\in C\}^\CLS\subseteq\Mult(C).
  \end{equation}
  The \(\Cst\)\nb-algebra \(D\) is called the Landstad\nb-Vaes algebra for the 
  \(\G\)\nb-product \((C,\gamma,i)\).   
  In particular, the third condition gives~\(\varphi\in\Mor(D,\hat{A}\otimes D)\). 
  Moreover, \(\hat{\beta}\defeq \flip\circ\varphi\in\Mor(D, D\otimes \hat{A})\) 
  is a \textup{(}right\textup{)} action of~\(\DuG\) on~\(D\), and extends to a
  \(\G\)\nb-equivariant isomorphism between~\(C\) and~\(D\rtimes\DuG\).
\end{theorem} 
This fundamental result was first proved by Vaes \cite{V2005}*{Theorem 6.7} (with slightly different conventions) for regular quantum groups~\(\G\) with the Haar weights and in~\cite{RW2018}*{Theorem 3.6 \& 3.8}  in the general setting of 
(not necessarily regular) \(\Cst\)\nb-quantum groups. However, we are going to restrict our attention to the regular \(\Cst\)\nb-quantum groups.

\subsection{Monoidal category of Yetter-Drinfeld C*-algebras}
\begin{definition}[\cite{NV2010}*{Definition 3.1}]
A~\(\G\)\nb-\emph{Yetter}\nb-\emph{Drinfeld} \(\Cst\)\nb-algebra 
is a triple \((C,\gamma,\widehat{\gamma})\) consisting of a 
\(\Cst\)\nb-algebra~\(C\)  along with actions 
\(\gamma\in\Mor(C,C\otimes A)\) and \(\widehat{\gamma}\in\Mor(C,C\otimes\hat{A})\) 
of \(\G\) and~\(\DuG\) that satisfy the \emph{Yetter}\nb-\emph{Drinfeld} 
compatibility criterion
\begin{equation}
 \label{eq:Yetter-Drinfeld}
   (\widehat{\gamma}\otimes\Id_{A})\gamma(c) 
= (\multunit_{23})\flip_{23}\Bigl((\gamma\otimes\Id_{\hat{A}})\widehat{\gamma}(c)\Bigr)(\multunit[*]_{23})
\qquad\text{for all \(c\in C\).}
\end{equation}
Indeed, \((C,\gamma,\widehat{\gamma})\) is a 
\(\G\)\nb-Yetter-Drinfeld \(\Cst\)\nb-algebra if and only 
if~\((C,\widehat{\gamma},\gamma)\) is a 
\(\DuG\)\nb-Yetter-Drinfeld \(\Cst\)\nb-algebra. 
\end{definition}
\begin{example}
 \label{ex:yett-drinf_reg}
 Let~\(\Qgrp{G}{A}\) be a regular \(\Cst\)\nb-quantum group. 
 Then~\(\Theta\colon A\to A\otimes\hat{A}\) defined by 
\(\Theta(a)\defeq \flip(\multunit[*](1_{\hat{A}}\otimes a)\multunit)\) 
for~\(a\in A\) is an action of~\(\DuG\) on~\(A\), and 
\((A,\Comult[A],\Theta)\) is a 
\(\G\)\nb-Yetter\nb-Drinfeld \(\Cst\)\nb-algebra 
(see~\cite{NV2010}*{Section 3}).
\end{example}

Let~\(\YDcat(\G)\) be the category with \(\G\)\nb-Yetter\nb-Drinfeld 
\(\Cst\)\nb-algebras as objects and~\(\G\) and~\(\DuG\)\nb-equivariant 
morphisms as arrows. Suppose, \((C_{1},\gamma_{1},\widehat{\gamma}_{1})\) and 
\((C_{2},\gamma_{2},\widehat{\gamma}_{2})\) are 
objects of~\(\YDcat(\G)\). 
Without loss of generality, suppose 
\((\corep{U}^{i},\varphi_{i})\) are faithful 
covariant representation of \((C_{i},\gamma_{i})\) on 
\(\Hils[L]_{i}\) and \((\corep{V}^{i},\widehat{\varphi}_{i})\) are
faithful covariant representations of~\((C_{i},\widehat{\gamma}_{i})\) on~\(\Hils[L]_{i}\) for 
\(i=1,2\), respectively. 

Define~\(\Braiding{\Hils[L]_{2}}{\Hils[L]_{1}}\colon \Hils[L]_{2}\otimes 
\Hils[L]_{1}\to\Hils[L]_{1}\otimes\Hils[L]_{2}\) by 
\(\Braiding{\Hils[L]_{2}}{\Hils[L]_{1}}\defeq Z\circ\Flip\), where  
\(Z\in\U(\Hils[L]_{1}\otimes\Hils[L]_{2})\) is the unique solution of~\eqref{eq:braid} 
for the pair of representations~\((\corep{U}^{1},\corep{V}^{2})\).

\begin{theorem}[\cite{MRW2014}*{Lemma
    3.20, Theorem 4.3, Theorem 4.9}]
  \label{the:crossed_prod}
  For~\(i=1,2\) define~\(j_{i}\in\Mor(C_{i},\Comp(\Hils[L]_{1}\otimes\Hils[L]_{2}))\) 
 by
\begin{equation}
 \label{eq:twisted_tensor}
 j_{1}(c_{1})\defeq \varphi_{1}(c_{1})\otimes 1_{\Hils[L]_{2}}, 
 \qquad
 j_{2}(c_{2})\defeq \Braiding{\Hils[L]_{2}}{\Hils[L]_{1}}(\varphi_{2}(c_{2})\otimes 1_{\Hils[L]_{1}}) \Dualbraiding{\Hils[L]_{1}}{\Hils[L]_{2}},
\end{equation}  
where~\(\Dualbraiding{\Hils[L]_{1}}{\Hils[L]_{2}}\defeq \Flip \circ Z^{*}\).
  Then the subspace~\(C_{1}\boxtimes C_{2} \defeq j_{1}(C_{1})j_{2}(C_{2})\)
  is a nondegenerate \(\Cst\)\nb-subalgebra of~\(\Bound(\Hils[L]_{1}\otimes\Hils[L]_{2})\) 
  and the triple~\((C_{1}\boxtimes C_{2},j_{1}, j_{2})\), up to equivalence, does not
  depend on the faithful covariant representations
  \((\corep{U}^{i},\varphi_{i})\) and~\((\corep{V}^{i},\varphi_{i})\) for 
  \(i=1,2\). 
\end{theorem}

Furthermore~\(C_{1}\boxtimes C_{2}\) becomes a~\(\G\)\nb-Yetter\nb-Drinfeld 
\(\Cst\)\nb-algebra with respect to the diagonal actions of~\(\G\) and \(\DuG\) 
defined by
\begin{alignat}{2}
 \label{eq:diag_G}
 & C_{1}\boxtimes C_{2}\ni x\overset{\gamma_{1}\bowtie\gamma_{2}}{\longrightarrow} 
 (\corep{U}^{1}\tenscorep \corep{U}^{2}) (x\otimes 1_{A}) (\corep{U}^{1}\tenscorep \corep{U}^{2})^{*} 
 \in C_{1}\boxtimes C_{2}\otimes A,\\
 \label{eq:diag_du_G}
 & C_{1}\boxtimes C_{2}\ni x\overset{\widehat{\gamma}_{1}\bowtie\widehat{\gamma}_{2}}{\longrightarrow}
 (\corep{V}^{1}\tenscorep \corep{V}^{2}) (x\otimes 1_{\hat{A}}) 
 (\corep{V}^{1}\tenscorep \corep{V}^{2})^{*}
 \in C_{1}\boxtimes C_{2}\otimes \hat{A}.
\end{alignat}
This following theorem has been proved in~\cite{NV2010}*{Section 3} 
in the presence of Haar weights on~\(\G\) and in 
\cite{MRW2016}*{Section 5} in the general framework of modular multiplicative unitaries.
\begin{theorem}
 \label{the:Yetter-Drinfeld_cat}
 \((\YDcat(\G),\boxtimes)\) is a monoidal category.  
\end{theorem}

\section{Landstad-Vaes algebra for the quantum groups with a projection}
\label{sec:Br-slices}
  Let~\(\Qgrp{H}{C}\) be a~\(\Cst\)\nb-quantum group, let 
  \(\DuQgrp{H}{C}\) be its dual, and 
  let~\(\multunit[C]\in\U(\hat{C}\otimes C)\) 
  be the reduced bicharacter of~\(\G[H]\).
  
  An element~\(\projbichar\in\U(\hat{C}\otimes C)\) 
is a \emph{projection} on~\(\G[H]\) if 
\begin{enumerate}
\item \(\projbichar\) is a quantum group endomorphism of~\(\G[H]\): 
\begin{equation}
 \label{eq:proj_bichar_cond}
 (\DuComult[C]\otimes\Id_{C})\projbichar =\projbichar_{23}\projbichar_{13}, 
 \qquad
 (\Id_{\hat{C}}\otimes\Comult[C])\projbichar=\projbichar_{12}\projbichar_{13},
\end{equation}
\item \(\projbichar\) is idempotent: 
for any~\(\G[H]\)\nb-Heisenberg pair~\((\alpha,\hat{\alpha})\) on~\(\Hils_{\alpha}\) 
\begin{equation}
 \label{eq:proj_cond}
  \projbichar_{\hat{\alpha}3}\projbichar_{1\alpha} 
  =\projbichar_{1\alpha}\projbichar_{13}\projbichar_{\hat{\alpha}3}
  \qquad\text{in~\(\U(\hat{C}\otimes\Comp(\Hils_{\alpha})\otimes C)\).}
\end{equation}
\end{enumerate}
The condition~\eqref{eq:proj_cond} above can be also formulated using \(\G[H]\)\nb-anti\nb-Heisenberg pairs.
Suppose~\((\bar{\alpha},\bar{\hat{\alpha}})\) is an~\(\G[H]\)\nb-anti\nb-Heisenberg pair on~\(\Hils_{\bar{\alpha}}\). Then~\cite{MRW2014}*{Proposition 3.9} implies that the representations~\(((\alpha\otimes\bar{\alpha})\circ\Comult[C],(\hat{\alpha}\otimes\bar{\hat{\alpha}})\circ\DuComult[C])\) of~\(C\) and~\(\hat{C}\) on~\(\Hils_{\alpha}\otimes\Hils_{\bar{\alpha}}\) commute. Subsequently, we have
\begin{align*}
&  \bigl((\Id_{\hat{C}}\otimes(\alpha\otimes\bar{\alpha})\Comult[C])\projbichar\bigr)_{123}
  \bigl(((\hat{\alpha}\otimes\bar{\hat{\alpha}})\DuComult[C]\otimes\Id_{C})\projbichar\bigr)_{234}
\\&=  \bigl(((\hat{\alpha}\otimes\bar{\hat{\alpha}})\DuComult[C]\otimes\Id_{C})\projbichar\bigr)_{234}    
\bigl((\Id_{\hat{C}}\otimes(\alpha\otimes\bar{\alpha})\Comult[C])\projbichar\bigr)_{123}
\end{align*}
in~\(\U(\hat{C}\otimes\Comp(\Hils_{\alpha}\otimes\Hils_{\bar{\alpha}})\otimes C)\). The conditions~\eqref{eq:proj_bichar_cond} simplifly the last equation as follows:
\[
     \projbichar_{1\alpha}\projbichar_{1\bar{\alpha}}
     \projbichar_{\bar{\hat{\alpha}}4}\projbichar_{\hat{\alpha}4}
  = \projbichar_{\bar{\hat{\alpha}}4}\projbichar_{\hat{\alpha}4}
     \projbichar_{1\alpha}\projbichar_{1\bar{\alpha}}.
\]
Commuting \(\projbichar^{*}_{1\alpha}\) with~\(\projbichar_{\bar{\hat{\alpha}}4}\) and~\(\projbichar^{*}_{\hat{\alpha}4}\) with~\(\projbichar_{1\bar{\alpha}}\), the last equation becomes 
\[
  \projbichar_{\bar{\hat{\alpha}}4}^{*}\projbichar_{1\bar{\alpha}}
  \projbichar_{\bar{\hat{\alpha}}4}\projbichar^{*}_{1\bar{\alpha}}
  =\projbichar_{1\alpha}^{*}\projbichar_{\hat{\alpha}4}
  \projbichar_{1\alpha}\projbichar^{*}_{\hat{\alpha}4}.
\]
 Hence~\eqref{eq:proj_cond} is equivalent to 
 \begin{equation}
  \label{eq:proj_cond_antiheis}
    \projbichar_{1\bar{\alpha}} \projbichar_{\bar{\hat{\alpha}}3}
  =\projbichar_{\bar{\hat{\alpha}}3}\projbichar_{13}\projbichar_{1\bar{\alpha}} 
  \qquad\text{in~\(\U(\hat{C}\otimes\Comp(\Hils_{\bar{\alpha}})\otimes C)\),}
 \end{equation}
 for any~\(\G[H]\)\nb-anti\nb-Heisenberg pair~\((\bar{\alpha},\bar{\hat{\alpha}})\) on~\(\Hils_{\bar{\alpha}}\).

Suppose~\((\alpha,\hat{\alpha})\) is an~\(\G[H]\)\nb-Heisenberg pair on~\(\Hils_{\alpha}\). Now \(\ProjBichar_{\alpha}\defeq(\hat{\alpha}\otimes\alpha)\projbichar\in\U(\Hils_{\alpha}\otimes\Hils_{\alpha})\) is a manageable mutliplicative unitary, which follows from \cite{MRW2017}*{Proposition 2.5}. Then the \(\Cst\)\nb-quantum group 
\(\Qgrp{G}{A}\) generated by~\(\ProjBichar_{\alpha}\), which does not depend on the choice of the~\(\G[H]\)\nb-Heisenberg pair~\((\alpha,\hat{\alpha})\), and~\(\projbichar\in\U(\hat{A}\otimes A)\subseteq\U(\hat{C}\otimes C)\). Then~\(\G\) is called the \emph{image} of \(\projbichar\).

In particular, we have \(A\subseteq\Mult(C)\). Moreover, the inclusion~\(i\colon A\hookrightarrow\Mult(C)\) is an element of~\(\Mor(A,C)\). 
To see this, once again, let us fix an~\(\G[H]\)\nb-Heisenberg pair~\((\alpha , \hat{\alpha})\) on~\(\Hils_{\alpha}\).   
Then~\eqref{eq:Comult_W} for~\(\Comult[C]\) is equivalent to 
\((\alpha\otimes\Id_{C})\Comult[C](c)=\multunit[C]_{\hat{\alpha}2}(\alpha(c)\otimes 1_{C})\multunit[C]_{\hat{\alpha}2}{ }^*\) for all~\(c\in C\). 
Consequently, the second condition in~\eqref{eq:proj_bichar_cond} is equivalent to~\(\projbichar_{1\alpha}^{*}\multunit[C]_{\hat{\alpha}3}
\projbichar_{1\alpha}=\projbichar_{13}\multunit[C]_{\hat{\alpha}3}\) in~\(\U(\hat{C}\otimes\Comp(\Hils_{\alpha})\otimes C)\). This implies 
\begin{align*}
  C &=\{(\omega_{1}\otimes\omega_{2}\otimes\Id_{C})(\projbichar_{1\alpha}^{*}\multunit[C]_{\hat{\alpha}3}\projbichar_{1\alpha})\mid \omega_{1}\in\hat{C}_{*},\omega_{2}\in\Bound(\Hils_{\alpha})_{*}\}^\CLS \\
     &=\{(\omega_{1}\otimes\omega_{2}\otimes\Id_{C})(\projbichar_{13}\multunit[C]_{\hat{\alpha}3})\mid \omega_{1}\in\hat{C}_{*},\omega_{2}\in\Bound(\Hils_{\alpha})_{*}\}^\CLS
       =AC .
\end{align*}
Since, \(C^*=C\) and~\(A^*=A\), we have~\(C=C^*=(AC)^{*}=C^{*}A^{*}=CA\).

 Now \cite{MRW2017}*{Proposition 2.8} shows that 
 \(\Qgrp{H}{C}\) with projection~\(\projbichar\in\U(\hat{C}\otimes C)\) with image 
 \(\Qgrp{G}{A}\) is equivalent to a quadruple~\((\G,\G[H],i,\Delta_{L})\) 
 consisting of \(\Cst\)\nb-quantum groups \(\Qgrp{G}{A}\), \(\Qgrp{H}{C}\) and
 morphisms \(i\in\Mor(A,C)\) and \(\Delta_{L}\in\Mor(C,A\otimes C)\) such that 
\begin{enumerate}
 \item \(i\) is a Hopf \Star{}homomorphism: \(\Comult[C]\circ i =(i\otimes i)\circ \Comult[A]\),
 \item \(\Delta_{L}\) is a left quantum group homomorphism: 
 \[
    (\Id_{A}\otimes\Comult[C])\circ\Delta_{L}=(\Delta_{L}\otimes\Id_{C})\circ\Comult[C], 
    \quad
    (\Comult[A]\otimes\Id_{C})\circ\Delta_{L}=(\Id_{A}\otimes\Delta_{L})\circ\Delta_{L}, 
 \]
 \item \((C,\Delta_{L},i)\) is a~\(\G\)\nb-product, that is, \((\Delta_{L},i)\) satisfy~\eqref{eq:landstad}.
\end{enumerate} 
In the next result we describe the Landstad-Vaes algebra for this \(\G\)\nb-product. For that matter 
we assume \(\G\) to be a regular \(\Cst\)\nb-quantum group.
\begin{proposition}
  \label{prop:Landstad_slices}
   Define 
  \(\brmultunit \defeq\projbichar^*\multunit[C]\in\U(\hat{C}\otimes C)\).   
  Then
  \[
  D\defeq \{(\omega\otimes\Id_C)\brmultunit\mid
  \omega\in \hat{C}_{*}\}^{\textup{CLS}} \subseteq \Mult(C).
  \]
  is the Landstad-Vaes algebra for the \(\G\)\nb-product 
  \((C,\Delta_{L},i)\). 
 \end{proposition}

First we prove the following technical lemma.
\begin{lemma}
  \label{lemm:comm_F_X}
  Let~\(\multunit\in\U(\hat{A}\otimes A)\) be the reduced 
  bicharacter of~\(\G\). Define \(X\in\U(\hat{A}\otimes
  C)\) by \(X \defeq (\Id_{\hat{A}}\otimes i)\multunit\). Then for 
  any \(\G[H]\)\nb-anti-Heisenberg pair \((\bar{\alpha},\bar{\hat{\alpha}})\)
  on a Hilbert space~\(\Hils_{\bar{\alpha}}\) we have the following 
  relation\textup{:}
  \begin{align}
    \brmultunit_{\bar{\hat{\alpha}}3}X_{13}X_{1\bar{\alpha}}
    &= X_{13}X_{1\bar{\alpha}}\brmultunit_{\bar{\hat{\alpha}}3}
    \qquad\text{in \(\U(\hat{A}\otimes\Comp(\Hils_{\bar{\alpha}})\otimes C)\).}
  \end{align}
\end{lemma}

\begin{proof}
  Since \((\bar{\alpha},\bar{\hat{\alpha}})\) is an \(\G[H]\)\nb-anti-Heisenberg
  pair,
  \begin{align}
    \label{eq:multunit_I_anti_Heis}
    \multunit[C]_{1\bar{\alpha}}\multunit[C]_{ \bar{\hat{\alpha}} 3}
    &=\multunit[C]_{\bar{\hat{\alpha}} 3}\multunit[C]_{13}\multunit[C]_{1 \bar{\alpha}}
    \qquad\text{in }\U(\hat{C}\otimes\Comp(\Hils_{\bar{\alpha}})\otimes C).
  \end{align}
   Combining~\eqref{eq:Comult_W} and~\eqref{eq:multunit_I_anti_Heis} for~\(\Comult[C]\) we 
  can show that 
  \begin{align}
   \label{eq:comult_anti_heis}
    (\Id_{C}\otimes\bar{\alpha})\Comult[C](c) 
    &=\flip(\multunit[C]_{\bar{\hat{\alpha}}2}{}^{*}(\bar{\alpha}(c)\otimes 1_{C})\multunit[C]_{\bar{\hat{\alpha}} 2})
    \qquad\text{for~\(c\in C\).}
 \end{align}
  The unitary \(X\defeq (\Id_{\hat{A}}\otimes i)\multunit\in
  \U(\hat{A}\otimes C)\) is a bicharacter (see~\cite{MRW2012}*{Definition 3.1}) because~\(i\) is a Hopf
  \Star{}homomorphism. So, in particular, \((\Id_{\hat{A}}\otimes\Comult[C])X =
  X_{12}X_{13}\) and it is equivalent to 
  \begin{align}
    \label{eq:char_cond_2ndleg_aHeis}
    X_{1\bar{\alpha}}\multunit[C]_{\bar{\hat{\alpha}}3}
    &=\multunit[C]_{\bar{\hat{\alpha}}3}X_{13}X_{1\bar{\alpha}}
    \qquad\text{in }\U(\hat{A}\otimes\Comp(\Hils_{\bar{\alpha}})\otimes C)
  \end{align}
  by~\eqref{eq:comult_anti_heis}. 

  The unitary \(\Duprojbichar\defeq \flip(\projbichar^{*})\in 
 \U(C\otimes\hat{C})\) is a projection on~\(\DuG[H]\). This 
 defines an injective Hopf \Star{}homomorphism \(\hat{i}\in\Mor(\hat{A},\hat{C})\) 
 such that~\(\projbichar=(\hat{i}\otimes i)\multunit\in\U(\hat{C}\otimes C)\). 
 Recall that~\(\projbichar\) satisfies~\eqref{eq:proj_cond_antiheis}. Since \(\hat{i}\) is injective, we may 
 apply~\(\hat{i}^{-1}\otimes\Id_{\Hils_{\bar{\alpha}}}\otimes\Id_{C}\) on the both sides of~\eqref{eq:proj_cond_antiheis} and obtain
    \begin{align}
    \label{eq:aux_G_anti_Heis}
    X_{1\bar{\alpha}}\projbichar_{\bar{\hat{\alpha}}3} =\projbichar_{\bar{\hat{\alpha}}3}X_{13}
    X_{1\bar{\alpha}}
    \qquad\text{in }\U(\hat{A}\otimes\Comp(\Hils_{\bar{\alpha}})\otimes C).
  \end{align}
  Subsequently, we complete the proof below using~\eqref{eq:char_cond_2ndleg_aHeis} and~\eqref{eq:aux_G_anti_Heis}:  
  \[
    \brmultunit_{\bar{\hat{\alpha}}3}X_{13}X_{1\bar{\alpha}}
    = \projbichar_{\bar{\hat{\alpha}}3}^*\multunit[C]_{\bar{\hat{\alpha}}3}X_{13}X_{1\bar{\alpha}}
    = \projbichar_{\bar{\hat{\alpha}}3}^*X_{1\bar{\alpha}}\multunit[C]_{\bar{\hat{\alpha}}3}
    = X_{13}X_{1\bar{\alpha}}\projbichar_{\bar{\hat{\alpha}}3}^*\multunit[C]_{\bar{\hat{\alpha}}3}\\
    = X_{13}X_{1\bar{\alpha}}\brmultunit_{\bar{\hat{\alpha}}3}.\qedhere
  \]
\end{proof}
\begin{proof}[Proof of Proposition~\textup{\ref{prop:Landstad_slices}}]
Let~\(\multunit\in\U(\hat{A}\otimes A)\) be the reduced bicharacter 
of~\(\G\). Recall the Hopf~\Star{}homomorphisms~\(i\colon A\to\Mult(C)\), 
\(\hat{i}\colon \hat{A}\to\Mult(\hat{C})\) and the bicharacter~\(X=(\Id_{\hat{A}}\otimes i)\multunit\) 
from the proof of Lemma~\ref{lemm:comm_F_X}. The bicharacter~\(X\) corresponds to the 
Hopf~\Star{}homomorphism~\(i\). Similarly, the bicharacter 
\(\chi\defeq (\hat{i}\otimes\Id_A)\multunit \in \U(\hat{C}\otimes A)\) corresponds to the Hopf \Star{}homomorphism \(\hat{i}\). 
These imply 
\begin{equation}
 \label{eq:Mar722}
(\Id_{\hat{C}}\otimes i)\chi=(\hat{i}\otimes i)\multunit=\projbichar
\in\U(\hat{C}\otimes C).
\end{equation} 
This shows that~\(\projbichar\) is the composition of bicharacters (see~\cite{MRW2012}*{Definition 3.5} ) 
viewed as the quantum group homomorphisms: \(\G[H]\stackrel{\chi}{\longrightarrow}\G\stackrel{X}{\longrightarrow}\G[H]\). More precisely, it is defined by
\begin{equation}
  \label{eq:bichar_comp}
  X_{\hat{\pi}3}\chi_{1\pi}=\chi_{1\pi}\projbichar_{13}X_{\hat{\pi}3} 
  \qquad
  \text{in~\(\U(\hat{C}\otimes\Comp(\Hils_{\pi})\otimes C)\),}
\end{equation}
where~\((\pi,\hat{\pi})\) is a~\(\G\)\nb-Heisenberg pair on 
the Hilbert space \(\Hils_{\pi}\). 

Suppose \(\Delta_{L}\in\Mor(C,A\otimes C)\) 
is the left quantum group homomorphism equivalent to~\(\chi\) 
given by~\cite{MRW2012}*{Theorem 5.5}
\begin{align}
  \label{eq:left_hom}
  (\Id_{\hat{C}}\otimes\Delta_L)\multunit[C]
  &=\chi_{12}\multunit[C]_{13}
  \qquad\text{in \(\U(\hat{C}\otimes A\otimes C)\)}.
\end{align}
The Landstad-Vaes algebra~\eqref{eq:explicit-landstad-vaes} 
for the~\(\G\)\nb-product~\((C,\Delta_{L},i)\) 
is given by 
\begin{equation}
 \label{eq:landstad-vaes_aux}
  D=\{(\omega\otimes\Id_{C})\varphi(c)\mid \omega\in\Bound(\Hils_{\pi})_{*}, c\in C\}^\CLS,
\end{equation}
where~\(\varphi(c)=X_{\hat{\pi}2}^{*}\Delta_{L}(c)_{\pi 2}X_{\hat{\pi}2}\) for all~\(c\in C\). Using~\eqref{eq:left_hom} and~\eqref{eq:bichar_comp} 
we have
\begin{align*}
  D
  &=\{(\omega'\otimes\omega\otimes\Id_{C})(X_{\hat{\pi}3}^{*}((\Id_{\hat{C}}\otimes\Delta_{L})\multunit[C])_{1\pi 3}X_{\hat{\pi}3})\mid \omega'\in\hat{C}_{*}, \omega\in\Bound(\Hils_{\pi})_{*}\}^\CLS \\
  &=\{(\omega'\otimes\omega\otimes\Id_{C})(X_{\hat{\pi}3}^{*}\chi_{1\pi}\multunit[C]_{13}X_{\hat{\pi}3})\mid \omega'\in\hat{C}_{*}, \omega\in\Bound(\Hils_{\pi})_{*}\}^\CLS  \\
  &=\{(\omega'\otimes\omega\otimes\Id_{C})(X_{\hat{\pi}3}^{*}\chi_{1\pi}\projbichar_{13}\brmultunit_{13}X_{\hat{\pi}3})\mid \omega'\in\hat{C}_{*}, \omega\in\Bound(\Hils_{\pi})_{*}\}^\CLS  \\
  &=\{(\omega'\otimes\omega\otimes\Id_{C})(\chi_{1\pi}X_{\hat{\pi}3}^{*}\brmultunit_{13}X_{\hat{\pi}3})\mid \omega'\in\hat{C}_{*}, \omega\in\Bound(\Hils_{\pi})_{*}\}^\CLS \\  
  &=\{(\omega'\otimes\omega\otimes\Id_{C})(X_{\hat{\pi}3}^{*}\brmultunit_{13}X_{\hat{\pi}3})\mid \omega'\in\hat{C}_{*}, \omega\in\Bound(\Hils_{\pi})_{*}\}^\CLS  \\
  &=\{(\omega'\otimes\omega\otimes\Id_{C})(X_{23}^{*}\brmultunit_{13}X_{23})\mid \omega'\in\hat{C}_{*}, \omega\in\hat{A}_{*}\}^\CLS  \\  
  &=\{(\omega\otimes\omega'\otimes\Id_{C})(X_{13}^{*}\brmultunit_{23}X_{13})\mid \omega\in\hat{A}_{*}, \omega'\in\hat{C}_{*}\}^\CLS \\
&=\{(\omega\otimes\omega'\otimes\Id_{C})(X_{13}^{*}\brmultunit_{\bar{\hat{\alpha}}3}X_{13})\mid \omega\in\hat{A}_{*}, \omega'\in\Bound(\Hils_{\bar{\alpha}})_{*}\}^\CLS ,  
\end{align*}
where \((\bar{\alpha},\bar{\hat{\alpha}})\) be an~\(\G[H]\)\nb-anti\nb-Heisenberg pair on the Hilbert space~\(\Hils_{\bar{\alpha}}\). Finally, using Lemma~\ref{lemm:comm_F_X} in the last computation we complete the proof below:
\begin{align*}
  D
  &=\{(\omega\otimes\omega'\otimes\Id_{C})(X_{1\bar{\alpha}}\brmultunit_{\bar{\hat{\alpha}}3}X^{*}_{1\bar{\alpha}})\mid \omega\in\hat{A}_{*}, \omega'\in\Bound(\Hils_{\bar{\alpha}})_{*}\}^\CLS \\
  &=\{(\omega'\otimes\Id_{C})\brmultunit_{\bar{\hat{\alpha}}2}\mid \omega'\in\Bound(\Hils_{\bar{\alpha}})_{*}\}^\CLS
    =\{(\omega'\otimes\Id_{C})\brmultunit\mid \omega'\in\hat{C}_{*}\}^\CLS . \qedhere
\end{align*}
\end{proof}
According to~\cite{MR2019}*{Theorem 2.18} an isomorphism between two~\(\Cst\)\nb-quantum groups~\(\G[H]_{1}=\Bialg{C_{1}}\) and~\(\G[H]_{2}=\Bialg{C_{2}}\) is 
a Hopf~\Star{}isomorphism \(f\in\Mor(C_{1},C_{2})\). Let~\(\projbichar_{k}\) be a projection on~\(\G[H]_{k}\) and let~\(\G_{k}=\Bialg{A_{k}}\) be the image of~\(\projbichar_{k}\) for~\(k=1,2\).
\begin{definition}
 \label{def:iso_qgrp_proj}
An \emph{isomorphism} between two~\(\Cst\)\nb-quantum groups with projections~\((\G[H]_{1},\projbichar_{1})\) and~\((\G[H]_{2},\projbichar_{2})\) is a Hopf~\Star{}isomorphism~\(f\in\Mor(C_{1}, C_{2})\) such that the restriction~\(f|_{A_{1}}\) is also Hopf~\Star{}isomorphism between~\(\G_{1}\) and~\(\G_{2}\).
\end{definition}
Let~\(\multunit[A_{1}]\in\U(\hat{A}_{1}\otimes A_{1})\) and~\(\multunit[A_{2}]\in\U(\hat{A}_{2}\otimes A_{2})\) be the reduced bicharacters of~\(\G_{1}\) and~\(\G_{2}\), respectively. Suppose~\(i_{k}\in\Mor(A_{k},C_{k})\) is the Hopf~\Star{}homomorphism induced by~\(\projbichar_{k}\) and~\(\hat{i}_{k}\in\Mor(\hat{A}_{k},\hat{C}_{k})\) be its dual satisfying~\eqref{eq:Mar722} for~\(k=1,2\). In particular, \((\hat{i}_{k}\otimes i_{k})\multunit[A_{k}]=\projbichar_{k}\) for~\(k=1,2\). The isomorphism~\(f\) in the Definition~\ref{def:iso_qgrp_proj} induces the Hopf~\Star{}isomorphism~\(f_{A}\in\Mor(A_{1},A_{2})\) such that~\(f\circ i_{1}=i_{2}\circ f_{A}\).

Let~\(\multunit[C_{1}]\in\U(\hat{C}_{1}\otimes C_{1})\) and~\(\multunit[C_{2}]\in\U(\hat{C}_{2}\otimes C_{2})\) be the reduced bicharacters of~\(\G[H]_1\) and~\(\G[H]_{2}\), respectively. Then a Hopf~\Star{}isomorphism \(f\in\Mor(C_{1},C_{2})\) is equivalent to the dual Hopf~\Star{}isomorphism~\(\hat{f}\in\Mor(\hat{C}_{1},\hat{C}_{2})\) which is characterised by the following equation:~\((\hat{f}\otimes f)\multunit[C_{1}]=\multunit[C_{2}]\in\U(\hat{C}_{1}\otimes C_{2})\). 
Now for any~\(\G[H]_{2}\)\nb-Heisenberg pair~\((\alpha,\hat{\alpha})\) on~\(\Hils\), the pair~\((\alpha\circ f,\hat{\alpha}\circ\hat{f})\) of representations of~\(C_{1}\) and~\(\hat{C}_{1}\) on~\(\Hils\) is an~\(\G[H]_{1}\)\nb-Heisenberg pair. Therefore~\((\hat{\alpha}\circ\hat{f}\otimes\alpha\circ f)\projbichar_{1}\) is a manageable multiplicative unitary and generates~\(\Bialg{f\circ i_{1}(A_{1})}\). By duality, the restriction~\(\hat{f}|_{\hat{A}_{1}}\) defines a Hopf~\Star{}isomorphism between~\(\DuG_{1}\) and~\(\DuG_{2}\) (inside~\(\DuG[H]_{1}\) and~\(\DuG[H]_{2}\)) and it is the dual of~\(f|_{A}\). Hence, we get the Hopf~\Star{}isomorphism~\(\hat{f}_{A}\in\Mor(\hat{A}_{1},\hat{A}_{2})\) such that~\(\hat{f}\circ \hat{i}_{1}=\hat{i}_{2}\circ \hat{f}_{A}\) and~\((\hat{f}_{A}\otimes f_{A})\multunit[A_{1}]=\multunit[A_{2}]\). Then~\eqref{eq:Mar722} gives 
\begin{equation}
 \label{eq:projiso-rel}
(\hat{f}\otimes f)\projbichar_{1} 
=(\hat{f}\circ\hat{i}_{1}\otimes f\circ i_{1})\multunit[A_{1}]
=(\hat{i}_{2}\circ \hat{f}_{A}\otimes i_{2}\circ f_{A} )\multunit[A_{1}]
=(\hat{i}_{2}\otimes i_{2})\multunit[A_{2}]
=\projbichar_{2}.
\end{equation}
Therefore, \((\hat{f}\otimes f)(\projbichar_{1}^{*}\multunit[C_{1}])=\projbichar_{2}^{*}\multunit[C_{2}]\).
Consequently, \(f\) defines an isomorphism between the Landstad\nb-Vaes algebras in Proposition~\ref{prop:Landstad_slices} for the~\(\G_{1}\)\nb-product and the \(\G_{2}\)\nb-product associated to~\((\G[H]_{1},\projbichar_{1})\) and~\((\G[H]_{2},\projbichar_{2})\) are also isomorphic. We shall use these facts later in Section~\ref{sec:boson}.

\section{From braided multiplicative unitaries to quantum groups with projection}
\label{sec:br_mult}
  Let~\(\Qgrp{G}{A}\) be a~\(\Cst\)\nb-quantum group, let 
  \(\DuQgrp{G}{A}\) be its dual, and 
  let~\(\multunit\in\U(\hat{A}\otimes A)\) 
  be the reduced bicharacter of~\(\G\).
 
 Let~\(\Hils[L]\) be a Hilbert space.  Consider a 
 pair of representations~\((\corep{U},\corep{V})\) of 
 \(\G\) and~\(\DuG\) on~\(\Hils[L]\) satisfying the commutation 
 relation
    \begin{equation}
      \label{eq:U_V_compatible_1}
      \corep{V}_{12}\corep{U}_{13} \multunit_{23}
      = \multunit_{23} \corep{U}_{13}\corep{V}_{12}
      \quad\text{in }\U(\Comp(\Hils[L])\otimes\hat{A}\otimes A).
    \end{equation}
 The pair~\((\corep{U},\corep{V})\) corresponds to a representation 
 of the quantum codouble~\(\Codouble{\G}\), the dual of the Drinfeld 
 double~\(\mathfrak{D}(\G)\), of~\(\G\) on~\(\Hils[L]\) and it is 
 called~\(\Codouble{\G}\)\nb-pair on~\(\Hils[L]\), see~\cite{R2015a}. In~\cite{MRW2016}*{Section 5}, it was observed that the representation category of~\(\Codouble{\G}\) is a unitarily braided monoidal (tensor product of representations of~\(\Codouble{\G}\)) category. We fix a \(\Codouble{\G}\)\nb-pair \((\corep{U},\corep{V})\) on~\(\Hils[L]\) and define~\(\Braiding{\Hils[L]}{\Hils[L]}\defeq Z\circ \Flip\), where 
 \(Z\in\U(\Hils[L]\otimes\Hils[L])\) is the solution of~\eqref{eq:braid}. In fact, \(\Braiding{\Hils[L]}{\Hils[L]}\) is the braiding isomorphism for the pair of objects \(((\corep{U},\corep{V}),(\corep{U},\corep{V}))\) in the representation category of~\(\Codouble{\G}\).

\begin{definition}[compare with~\cite{MRW2017}*{Definition 3.2}]
  \label{def:braided_multiplicative_unitary}
  A \emph{braided multiplicative unitary on~\(\Hils[L]\) over 
  \(\G\) relative to \((\corep{U},\corep{V})\)} is a unitary 
  \(\BrMultunit\in\U(\Hils[L]\otimes\Hils[L])\) 
  such that 
  \begin{enumerate}
  \item \(\BrMultunit\) is \emph{invariant} with respect to the
    tensor product representation \(\corep{U} \tenscorep \corep{U} \defeq
    \corep{U}_{13}\corep{U}_{23}\) of~\(\G\)
    on~\(\Hils[L]\otimes\Hils[L]\):
    \begin{equation}
      \label{eq:F_U-invariant}
      \corep{U}_{13} \corep{U}_{23} \BrMultunit_{12}
      = \BrMultunit_{12} \corep{U}_{13} \corep{U}_{23}
      \quad\text{in }\U(\Comp(\Hils[L]\otimes\Hils[L])\otimes A);
    \end{equation}
  \item \(\BrMultunit\) is \emph{invariant} with respect to the 
    tensor product representation \(\corep{V} \tenscorep \corep{V} \defeq
    \corep{V}_{13}\corep{V}_{23}\) of~\(\DuG\)
    on~\(\Hils[L]\otimes\Hils[L]\):
    \begin{equation}
      \label{eq:F_V-invariant_1}
      \corep{V}_{13} \corep{V}_{23} \BrMultunit_{12}
      = \BrMultunit_{12} \corep{V}_{13} \corep{V}_{23}
      \quad\text{in }\U(\Comp(\Hils[L]\otimes\Hils[L])\otimes\hat{A});
    \end{equation}
  \item \(\BrMultunit\) satisfies the \emph{braided pentagon equation}~\eqref{eq:top-braided_pentagon}.
  \end{enumerate}
 \end{definition}

 Let \((\pi,\hat{\pi})\) be the \(\G\)\nb-Heisenberg
 pair on~\(\Hils\) coming from a manageable multiplicative
 unitary \(\Multunit\in\U(\Hils\otimes\Hils)\) generating~\(\G\), 
 that is, \((\hat{\pi}\otimes\pi)\multunit=\Multunit\). Using 
 it, we define the unitaries 
 \(\Ducorep{V}\in\U(\hat{A}\otimes\Comp(\Hils[L]))\), 
 \(\Corep{U},\Corep{V}\in\U(\Hils[L]\otimes\Hils)\) and 
 \(\DuCorep{V}\in\U(\Hils\otimes\Hils[L])\) by 
 \[
   \Ducorep{V}\defeq \flip(\corep{V}^{*}), 
   \quad
   \Corep{U}\defeq (\Id_{\Hils[L]}\otimes\pi)\corep{U}, 
   \quad 
  \Corep{V}\defeq (\Id_{\Hils[L]}\otimes\hat{\pi})\corep{V}, 
  \quad
  \DuCorep{V}\defeq \Flip \Corep{V}^{*}\Flip=(\hat{\pi}\otimes\Id_{\Hils[L]})\Ducorep{V}. 
 \]
Then~\eqref{eq:braid} and~\eqref{eq:U_V_compatible_1}
 for~\(\corep{U}\) and \(\corep{V}\) are equivalent to
 \begin{alignat}{2} 
    \label{eq:braiding}
      Z_{13} &= \DuCorep{V}_{23} \Corep{U}_{12}^*
      \DuCorep{V}_{23}^* \Corep{U}_{12}
      &\quad\text{in \(\U(\Hils[L]\otimes\Hils\otimes\Hils[L])\);}\\
       \label{eq:U_V_compatible}
      \Corep{U}_{23} \Multunit_{13} \DuCorep{V}_{12}
      &= \DuCorep{V}_{12} \Multunit_{13} \Corep{U}_{23}
      &\quad\text{in \(\U(\Hils\otimes\Hils[L]\otimes\Hils)\).}
\end{alignat}
 Now~\(\BrMultunit\) gives rise to a pair of multiplicative unitaries
\(\Multunit[C], \ProjBichar\in\U(\Hils\otimes\Hils[L]\otimes\Hils\otimes\Hils[L])\) 
given by \cite{MRW2017}*{Theorem 3.7}:
  \begin{alignat}{2}
    \label{eq:big_MU}
    \Multunit[C] &\defeq
    \Multunit_{13}\Corep{U}_{23}\DuCorep{V}^*_{34}\BrMultunit_{24}\DuCorep{V}_{34}
    &\qquad\text{in }\U(\Hils\otimes\Hils[L]\otimes\Hils\otimes\Hils[L]),\\
    \label{eq:Proj_Unit}
    \ProjBichar &\defeq \Multunit_{13}\Corep{U}_{23}
    &\qquad\text{in }\U(\Hils\otimes\Hils[L]\otimes\Hils\otimes\Hils[L]).
 \end{alignat}   
Suppose~\(\BrMultunit\) is manageable in the sense of~\cite{MRW2017}*{Definition 3.5}. Then we translate it to the manageability of the ordinary multiplicative unitary \(\Multunit[C]\) using~\cite{MRW2017}*{Theorem 3.8}. On the other hand, manageability of~\(\ProjBichar\) follows from the manageability of~\(\Multunit\). Let  \(\Qgrp{H}{C}\) be the \(\Cst\)\nb-quantum group generated by~\(\Multunit[C]\). Then~\cite{MRW2017}*{Theorem 3.7} and 
 \cite{MRW2012}*{Lemma 3.2} imply 
 \(\ProjBichar\in\U(\hat{C}\otimes C)\) is a projection on~\(\G[H]\).  In the 
 next lemma we ensure that the image of~\(\ProjBichar\) is~\(\G\).
\begin{lemma}
  \label{lemm:rep_hat-A}
  Let~\((\pi,\hat{\pi})\) be a \(\G\)\nb-Heisenberg pair on 
  \(\Hils\). There is a faithful representation \(\hat{\rho}\colon\hat{A}\to
  \Bound(\Hils\otimes\Hils[L])\) such that
  \((\hat{\rho}\otimes\pi)\multunit =
  \Multunit_{12}\Corep{U}_{13}\in\U(\Hils\otimes\Hils[L]\otimes\Hils)\). 
  Define \(\rho\colon A\to\Bound(\Hils\otimes\Hils[L])\) by 
  \(\rho(a)\defeq \pi(a)\otimes 1\). Then~\(\ProjBichar=(\hat{\rho}\otimes\rho)
  \multunit\) and generates the \(\Cst\)\nb-quantum group~\(\Qgrp{G}{A}\). In 
  particular, \(\G\) is a Woronowicz closed quantum subgroup of~\(\G[H]\) as 
  in~\textup{\cite{DKSS2012}*{Definition 3.2}}.
\end{lemma}
\begin{proof}
  Let \((\bar{\alpha},\bar{\hat{\alpha}})\) be a \(\G\)\nb-anti-Heisenberg pair on a
  Hilbert space~\(\Hils_{\bar{\alpha}}\).  By virtue 
  of~\eqref{eq:char_cond_2ndleg_aHeis}, the 
  equation~\eqref{eq:corep_cond} for~\(\corep{U}\) is equivalent to
  \[
  \corep{U}_{1\bar{\alpha}}\multunit_{\bar{\hat{\alpha}} 3}
  = \multunit_{\bar{\hat{\alpha}} 3}\corep{U}_{13}\corep{U}_{1\bar{\alpha}}
  \qquad\text{in }\U(\Comp(\Hils[L]\otimes\Hils_{\bar{\alpha}})\otimes A),
  \]
   Applying~\(\flip_{12}\) on both sides of the last equation and rearranging unitaries we obtain
  \begin{equation}
   \label{eq:faithful_rep_proj}
  \Ducorep{U}_{\bar{\alpha} 2}^*\multunit_{\bar{\hat{\alpha}} 3} \Ducorep{U}_{\bar{\alpha} 2}
  = \multunit_{\bar{\hat{\alpha}} 3}\corep{U}_{23}
  \qquad\text{in } \U(\Comp(\Hils_{\bar{\alpha}}\otimes\Hils[L])\otimes A).
 \end{equation}
  Here \(\Ducorep{U}\defeq \flip(\corep{U}^*)\in
  \U(A\otimes\Comp(\Hils[L]))\). Define a faithful 
  representation~\(\hat\rho'\colon \hat{A}\to\Bound(\Hils_{\bar{\alpha}}\otimes\Hils[L])\) 
  by \(\hat{\rho}'(\hat{a})
  \defeq \Ducorep{U}^*_{\bar{\alpha} 2}(\bar{\hat{\alpha}}(\hat{a})\otimes 1)
  \Ducorep{U}_{\bar{\alpha} 2}\).  The right hand side of~\eqref{eq:faithful_rep_proj} 
  implies the first component of~\(\hat{\rho}'(\hat{A})\) is inside the image of~\(\bar{\hat{\alpha}}\). 
  Also the representations~\(\bar{\hat{\alpha}}\), \(\hat{\pi}\) are faithful by~\cite{R2015a}*{Proposition 3.2}. 
  These allow to define the desired 
  representation~\(\hat{\rho}\colon \hat{A}\to\Bound(\Hils\otimes\Hils[L])\) by 
  \(\hat{\rho}(\hat{a}) \defeq
  (\hat{\pi}\circ\bar{\hat{\alpha}}^{-1}\otimes\Id_{\Hils[L]}) 
  \hat{\rho}'(\hat{a})\).  Then~\(\hat{\rho}\) is faithful and satisfies
  \((\hat{\rho}\otimes\pi)\multunit =
  \Multunit_{13}\Corep{U}_{23}\) by~\eqref{eq:faithful_rep_proj}. 
  Since, \(\ProjBichar=(\hat{\rho}\otimes\rho)\multunit\) is a manageable multiplicative unitary and 
  \(\rho\) is a faithful representation of~\(A\) on~\(\Bound(\Hils\otimes\Hils[L])\), we have 
  \(\pi(A)\otimes 1_{\Hils[L]}=\{(\omega\otimes\Id_{\Hils\otimes\Hils[L]})\ProjBichar\mid \omega\in\Bound(\Hils\otimes\Hils[L])_{*}\}\). 
  Finally, a simple computation using Theorem~\ref{the:Cst_quantum_grp_and_mult_unit} shows that,~\(\ProjBichar\) implements the comultiplication 
  map~\(\Comult[A]\) on~\(\rho(A)\): \((\rho\otimes\rho)\Comult[A](a)=\ProjBichar (\rho(a)\otimes 1)\ProjBichar^{*}\) for all~\(a\in A\).
\end{proof}
  Let us identify \(C\), \(\hat{C}\) with 
  their images inside~\(\Bound(\Hils\otimes\Hils[L]\otimes\Hils
  \otimes\Hils[L])\) under the representations obtained from 
  the \(\G[H]\)\nb-Heisenberg pair 
  that arises from the manageable multiplicative unitary 
  \(\Multunit[C]\) in~\eqref{eq:big_MU}.
 We also notice that the images of~\(\rho\) and~\(\hat{\rho}\) are contained 
 inside the images of~\(C\) and~\(\hat{C}\) in  
 \(\Bound(\Hils\otimes\Hils[L])\), respectively.
  
  \begin{proposition}
   \label{prop:G-prod_proj}
    The unitary \(\chi\defeq (\hat{\rho}\otimes\Id_{A})\multunit\in 
    \U(\hat{C}\otimes A)\) is a bicharacter from 
    \(\G[H]\) to~\(\G\). Suppose~\(\Delta_{L}\in\Mor(C, A\otimes C)\) is the 
    left quantum group homomorphism associated to~\(\chi\). 
    The \(\Cst\)\nb-quantum group~\(\Qgrp{H}{C}\) with projection 
    \(\ProjBichar\) with image~\(\G\) is equivalent to the quadruple 
    \((\G,\G[H],\rho,\Delta_{L})\) described 
    in~\textup{\cite{MRW2017}*{Proposition 2.8}}.
  \end{proposition}
  \begin{proof}  
  Recall~\(\ProjBichar=(\hat{\rho}\otimes\rho)\multunit\in\U(\hat{C}\otimes C)\) and, 
  in particular, \(\rho\in\Mor(A,C)\) is faithful. 
  Then the first condition in~\eqref{eq:proj_bichar_cond} and 
  \eqref{eq:dual_Comult_W} together give  
  \[
    (\DuComult[C]\circ\hat{\rho}\otimes\rho)\multunit   
    =(\DuComult[C]\otimes\Id_{C})\ProjBichar   
    =\ProjBichar_{23}\ProjBichar_{13}
    =((\hat{\rho}\otimes\hat{\rho})\circ\DuComult[A]\otimes\rho)\multunit.
  \]
 Taking slices on the third leg of the 
 last expression by~\(\omega\in C'\) shows that 
 \(\hat{\rho}\in\Mor(\hat{A},\hat{C})\) is a 
 Hopf \Star{}homomorphism. Similarly, we can prove that 
 \(\rho\in\Mor(A,C)\) is also a Hopf~\Star{}homomorphism.
 
 Thus~\(\chi\defeq (\hat{\rho}\otimes\Id_{A})\multunit\in 
 \U(\hat{C}\otimes A)\) is a bicharacter from~\(\G[H]\) to~\(\G\) and 
 the composition~\(\G[H]\to\G\to\G[H]\) is the 
 bicharacter \((\Id_{\hat{C}}\otimes\rho)\chi=\ProjBichar\).
 
 Let~\(\Delta_{R}\in\Mor(C,C\otimes A)\) be the right 
 quantum group homomorphism equivalent to~\(\chi\). 
 Then~\cite{MRW2012}*{Theorem 5.3} 
 and Lemma~\ref{lemm:rep_hat-A} imply
 \[
   (\Id_{\hat{A}}\otimes\Delta_{R}\circ\rho)\multunit
   =\chi_{23}\multunit_{1\rho}\chi_{23}^{*}
   =\multunit_{\hat{\rho}3}\multunit_{1\rho}\multunit_{\hat{\rho}3}
   =\multunit_{1\rho}\multunit_{13}
   =(\Id_{\hat{A}}\otimes(\rho\otimes\Id_{A})\circ\Comult[A])\multunit .
 \]
 Taking slices on the first leg of the 
 last expression by~\(\omega\in\hat{A}'\) gives 
 \(\Delta_{R}\circ\rho=(\rho\otimes\Id_{A})\circ\Comult[A]\). 
 Finally, \((\Delta_{L},\rho)\) is equivalent to~\((\Delta_{R},\rho)\) 
 \cite{MRW2017}*{Proposition 2.8} and 
 \((\Delta_{L},\rho)\) satisfies~\eqref{eq:landstad}.
\end{proof}

\section{The main results} 
 \label{sec:brd_Cstar_qgrp}
 Borrowing the same notations from the last section we state and 
 prove the first main result of this article.
\begin{theorem}
  \label{the:braid_Qgrp}
  Suppose \(\BrMultunit\in\U(\Hils[L]\otimes\Hils[L])\) is a manageable
  braided multiplicative unitary over a regular 
  \(\Cst\)\nb-quantum group 
  \(\Qgrp{G}{A}\) relative to~\((\corep{U},\corep{V})\).  
  Define
  \begin{equation}
   \label{eq:slice_first_brmult}
   B \defeq \{(\omega\otimes\Id_{\Hils[L]})\BrMultunit \mid
    \omega\in\Bound(\Hils[L])_*\}^{\textup{CLS}}, 
    \quad 
  \Comult[B](b)\defeq \BrMultunit (b\otimes 1_{\Hils[L]})\BrMultunit^{*}  
  \quad
  \text{for all~\(b\in B\).}
  \end{equation}
  Then
  \begin{enumerate}
  \item \(B\) is a
    nondegenerate, separable \(\Cst\)\nb-subalgebra of~\(\Bound(\Hils[L])\);
  \item The elements \(\beta\in\Mor(B,\Comp(\Hils[L])\otimes A)\) and
    \(\hat{\beta}\in\Mor(B,\Comp(\Hils[L])\otimes\hat{A})\) defined by
    \begin{equation}
      \label{eq:YD_on_B}
      \beta(b)\defeq \corep{U}(b\otimes 1)\corep{U}^* ,
      \qquad
      \hat{\beta}(b)\defeq \corep{V}(b\otimes 1)\corep{V}^*
    \end{equation}
    are actions of \(\G\) and~\(\DuG\) on~\(B\). Moreover, 
    \((B,\beta,\widehat{\beta})\) is an object of the category~\(\YDcat(\G)\).
 \item\label{eq:multiplier_F} \(\BrMultunit\in\U(\Comp(\Hils[L])\otimes B)\);
 \suspend{enumerate}
 Suppose \(j_1,j_2\in\Mor(B,B\boxtimes B)\) be the canonical
    morphisms in~\eqref{eq:twisted_tensor}. 
 \resume{enumerate}   
 \item Then~\(\Comult[B]\) is the unique arrow~\(B\to B\boxtimes B\) in the 
 category~\(\YDcat(\G)\) characterised by
    \begin{equation}
      \label{eq:brd_chr_leg2}
      (\Id_{\Hils[L]}\otimes\Comult[B])\BrMultunit
      = (\Id_{\Hils[L]}\otimes j_{1})\BrMultunit 
      (\Id_{\Hils[L]}\otimes j_{2})\BrMultunit
      \qquad\text{in }\U(\Comp(\Hils[L])\otimes B\boxtimes B).
    \end{equation}
    Furthermore, \(\Comult[B]\) is coassociative \textup{:}
    \begin{equation}
      \label{eq:br_coasso}
      (\Id_B\boxtimes\Comult[B])\circ\Comult[B]
      = (\Comult[B]\boxtimes\Id_B)\circ\Comult[B],
    \end{equation}
    and satisfies the cancellation conditions:
    \begin{equation}
      \label{eq:br_cancel}
      j_1(B)\Comult[B](B)=B\boxtimes B=\Comult[B](B)j_2(B).
    \end{equation}
  \end{enumerate}    
\end{theorem}

\begin{proof} 
 The image of~\(\ProjBichar\) is \(\Qgrp{G}{A}\) by 
 Lemma~\ref{lemm:rep_hat-A} and \(\G\) is regular 
 by assumption. Then we apply Proposition~\ref{prop:Landstad_slices}  
 for the \(\G\)\nb-triple~\((C,\Delta_{L},\rho)\) constructed 
 in Proposition~\ref{prop:G-prod_proj}.

 \textbf{Part (1):}~Since 
  \begin{align*}
  D &=\{(\omega'\otimes\omega\otimes\Id_{\Hils\otimes\Hils[L]})\ProjBichar^*\Multunit[C]\mid
  \text{\(\omega'\in\Bound(\Hils)_*\), \(\omega\in\Bound(\Hils[L])_{*}\)}\}^\CLS \\
  &=\{(\omega\otimes\Id_{\Hils\otimes\Hils[L]})\DuCorep{V}^*_{23}\BrMultunit_{13}
  \DuCorep{V}_{23}\mid \omega\in\Bound(\Hils[L])_* \}^{\textup{CLS}}\\ 
  &=\DuCorep{V}^* \bigl(1_{\Hils} \otimes\{(\omega\otimes\Id_{\Hils[L]})\BrMultunit\mid \omega\in\Bound(\Hils[L])_* \}^{\textup{CLS}}\bigr) \DuCorep{V}
  \end{align*}
  is a \(\Cst\)\nb-algebra by Proposition~\ref{prop:Landstad_slices}; hence so is 
  \(B\defeq \{(\omega\otimes\Id_{\Hils[L]})\BrMultunit\mid
  \omega\in\Bound(\Hils[L])_* \}^{\textup{CLS}}\subseteq
  \Bound(\Hils[L])\).
  
  Furthermore, the second condition in Theorem~\ref{the:landstad} gives 
  \(DC=C\). Also \(C\Comp(\Hils\otimes\Hils[L])=\Comp(\Hils\otimes 
  \Hils[L])\) because \(C\) is constructed from the manageable multiplicative 
  unitary~\(\Multunit[C]\), and 
  \(\DuCorep{V}\in\U(\Hils\otimes\Hils[L])\). Therefore, 
  \begin{multline*}
   (1_{\Hils}\otimes B)\Comp(\Hils\otimes\Hils[L]) 
   =\DuCorep{V}D\DuCorep{V}^{*}\Comp(\Hils\otimes\Hils[L]) 
   =\DuCorep{V}D\Comp(\Hils\otimes\Hils[L])
   =\DuCorep{V}DC\Comp(\Hils\otimes\Hils[L])\\
   =\DuCorep{V}C\Comp(\Hils\otimes\Hils[L])
   =\DuCorep{V}\Comp(\Hils\otimes\Hils[L])
   =\Comp(\Hils\otimes\Hils[L]).
  \end{multline*}
  Thus \(B\) acts nondegenerately on~\(\Hils[L]\). Separability 
  of~\(B\) follows from the separability of~\(\Bound(\Hils[L])_{*}\).

  \textbf{Part (2):}~Define \(\hat{\beta}(b)\defeq \corep{V}(b\otimes
  1_{\hat{A}})\corep{V}^*\) for \(b\in B\).  Clearly, \(\hat\beta\) is
  injective. The unitary~\(X\) in Lemma~\ref{lemm:comm_F_X}  
  is \((\Id_{\hat{A}}\otimes\rho)\multunit=\Multunit_{12}\) and 
  third condition in Theorem~\ref{the:landstad} becomes
  \begin{equation} 
   \label{eq:Poldes_Ducoat-aux}
  \hat{\pi}(\hat{A})\otimes \DuCorep{V}^*(1_{\Hils}\otimes B)\DuCorep{V}
  = (\hat{\pi}(\hat{A})\otimes 1_{\Hils\otimes\Hils[L]})
  \Multunit[*]_{12}\DuCorep{V}^*_{23}
  (1_{\Hils}\otimes 1_{\Hils[L]}\otimes B)
  \DuCorep{V}_{23}\Multunit_{12}.
  \end{equation}
  Now the condition~\eqref{eq:corep_cond} for the 
  representation~\(\corep{V}\) is equivalent to
   \begin{align*}
     \DuCorep{V}_{23} \Multunit_{12}
      &= \Multunit_{12} \DuCorep{V}_{13} \DuCorep{V}_{23}
      \qquad\text{in }\U(\Hils\otimes\Hils\otimes\Hils[L]).
   \end{align*}
  Using it we simplify the right hand side of~\eqref{eq:Poldes_Ducoat-aux}:
  \begin{align*}
  & (\hat{\pi}(\hat{A})\otimes 1_{\Hils\otimes\Hils[L]})
  \Multunit[*]_{12}\DuCorep{V}^*_{23}
  (1_{\Hils}\otimes 1_{\Hils[L]}\otimes B)
  \DuCorep{V}_{23}\Multunit_{12}\\
  & =(\hat{\pi}(\hat{A})\otimes 1_{\Hils\otimes\Hils[L]})
    \DuCorep{V}^*_{23}\DuCorep{V}^*_{13}\Multunit[*]_{12}
    (1_{\Hils}\otimes 1_{\Hils[L]}\otimes B)
    \Multunit_{12}\DuCorep{V}_{13}\DuCorep{V}_{23}\\
    &=(\hat{\pi}(\hat{A})\otimes 1_{\Hils\otimes\Hils[L]})
    \DuCorep{V}^*_{23}\DuCorep{V}^*_{13}
    (1_{\Hils}\otimes 1_{\Hils[L]}\otimes B)\DuCorep{V}_{13}\DuCorep{V}_{23}.
  \end{align*}
  Injectivity of~\(\hat{\pi}\) implies
  \begin{equation}
    \label{eq:Poldes_Ducoat}
    \hat{A}\otimes B
    = (\hat{A}\otimes 1_{\Hils[L]})
    \Ducorep{V}^*(1_{\Hils}\otimes B)\Ducorep{V}.
  \end{equation}
  This is equivalent to Podle\'s condition (up to~\(\flip\)).
  Thus \(\hat{\beta}\in\Mor(B,B\otimes \hat{A})\) and the 
  condition~\eqref{eq:corep_cond} for 
  \(\corep{V}\) yields~\eqref{eq:right_action} for~\(\hat{\beta}\)

  Similarly, \(\beta(b)\defeq \corep{U}(b\otimes 1_A)\corep{U}^*\)
  is injective, and it is sufficient to establish the Podle\'s condition 
  for~\(\beta\). Then
  \((B,\beta,\hat\beta)\) will become a \(\G\)\nb-Yetter\nb-Drinfeld
  \(\Cst\)\nb-algebra because the unitaries 
  \(\Corep{U},\DuCorep{V}\) satisfy the commutation 
  relation~\eqref{eq:U_V_compatible_1}.

  By virtue of the second condition in Theorem~\ref{the:landstad} 
  \(C=\rho(A)D=(\pi(A)\otimes 1_{\Hils[L]}) \DuCorep{V}^*(1_{\Hils}\otimes
  B)\DuCorep{V}\). Recall the right quantum group homomorphism 
  \(\Delta_R\in\Mor(C,C\otimes A)\) equivalent to the 
  bicharacter~\(\chi=(\hat{\rho}\otimes\pi)\multunit=\Multunit_{13}\Corep{U}_{23}\) 
  in Proposition~\ref{prop:G-prod_proj}. In particular, \(\Delta_{R}\) is an 
  action of~\(\G\) on~\(C\), see \cite{MRW2012}*{Lemma 5.8}. Combining 
  the Podle\'s condition for~\(\Delta_{R}\) and~\cite{MRW2012}*{Equation (33)} 
  we get
  \begin{multline*}
  (\pi(A)\otimes 1_{\Hils[L]\otimes\Hils}) 
  \DuCorep{V}^*_{12} (1_{\Hils}\otimes B\otimes 1_{\Hils}) 
  \DuCorep{V}_{12}\Corep{U}_{23}^*
  \Multunit[*]_{13}(1_{\Hils\otimes\Hils[L]}\otimes\pi(A))
  \\= \Corep{U}_{23}^*\Multunit[*]_{13}
    (\pi(A)\otimes 1_{\Hils[L]\otimes\Hils}) 
    \DuCorep{V}^*_{12} (1_{\Hils}\otimes B\otimes 1_{\Hils}) 
    \DuCorep{V}_{12}(1_{\Hils\otimes\Hils[L]}\otimes\pi(A)) .
  \end{multline*}
  Multiplying~\(\Comp(\Hils)\) to the first leg from left and 
  right of the last equation 
  and using the nondegeneracy of \(\pi\), that is \(\pi(A)\Comp(\Hils)=\Comp(\Hils)\), 
  we obtain
  \begin{multline*}
    (\Comp(\Hils)\otimes 1_{\Hils[L]\otimes\Hils}) 
    \DuCorep{V}^*_{12}(1_{\Hils}\otimes B\otimes 1_{\Hils}) 
    \DuCorep{V}_{12}\Corep{U}_{23}^*
    \Multunit[*]_{13}(\Comp(\Hils)\otimes 1_{\Hils[L]}\otimes\pi(A))
    \\=
    (\Comp(\Hils)\otimes 1_{\Hils[L]\otimes\Hils}) \Corep{U}_{23}^*\Multunit[*]_{13}
    (\pi(A)\otimes 1_{\Hils[L]\otimes\Hils})\DuCorep{V}^*_{12}
    (1_{\Hils}\otimes B\otimes 1_{\Hils}) \DuCorep{V}_{12}
    (\Comp(\Hils)\otimes 1_{\Hils[L]}\otimes\pi(A)).
  \end{multline*}
  Similarly, the nondegeneracy of~\(\hat{\pi}\) 
  and~\eqref{eq:Poldes_Ducoat} together imply
  \begin{multline*}
  (\Comp(\Hils)\otimes B\otimes 1_{\Hils})\Corep{U}_{23}^*
  \Multunit[*]_{13}(\Comp(\Hils)\otimes 1_{\Hils[L]}\otimes\pi(A))
  \\= (\Comp(\Hils)\otimes 1_{\Hils[L]\otimes\Hils})
  \Corep{U}_{23}^*\Multunit[*]_{13}
  (\pi(A)\Comp(\Hils)\otimes B\otimes \pi(A)).
  \end{multline*}
  Next we apply Theorem~\ref{the:Cst_quantum_grp_and_mult_unit}~(2),  
  that is \(\Multunit(\Comp(\Hils)\otimes\pi(A))=\Comp(\Hils)\otimes\pi(A)\),
  to simplify the last equation   
  \[
  \Comp(\Hils)\otimes \bigl((B\otimes 1_{\Hils})\Corep{U}^*
  (1_{\Hils[L]}\otimes\pi(A)\bigr)
  = \Comp(\Hils)\otimes \bigl(\Corep{U}^*
     (B\otimes\pi(A)\bigr).
  \]
  Finally,  taking slices by \(\omega\in\Bound(\Hils)_{*}\) on the first leg and
  and then multiplying the last equation by~\(\Corep{U}\) from the left, we obtain
  \[
  \Corep{U} (B\otimes 1_{\Hils}) \Corep{U}^*
  (1_{\Hils[L]}\otimes\pi(A)) = B\otimes \pi(A).
  \]
 This is equivalent to the Podle\'s condition for~\(\beta\) as~\(\pi\) is injective.
 
  \textbf{Part (3):}~Once again, recall the 
  second condition in the Landstad theorem~\ref{the:landstad}:
  \(C=(\pi(A)\otimes 1_{\Hils[L]})\DuCorep{V}^{*}(1\otimes B) 
  \DuCorep{V}\subset\Bound(\Hils\otimes\Hils[L])\). 
  Since \(\Multunit[C]\) is a unitary multiplier of 
  \(\Comp(\Hils\otimes\Hils[L])\otimes C\) we have 
  \((\Comp(\Hils)\otimes\Comp(\Hils[L])\otimes C)\Multunit[C]= 
  \Comp(\Hils)\otimes\Comp(\Hils[L])\otimes C\). 

Equivalently, 
  \[
   \DuCorep{V}^{*}_{34}B_{4}\DuCorep{V}_{34}\Comp(\Hils)_{1}
   \Comp(\Hils[L])_{2}\pi(A)_{3}\Multunit_{13}\Corep{U}_{23}
   \DuCorep{V}_{34}^{*}\BrMultunit_{24}\DuCorep{V}_{34}
   =\Comp(\Hils)_{1}\Comp(\Hils[L])_{2} \pi(A)_{3}\DuCorep{V}^{*}_{34}B_{4}\DuCorep{V}_{34}.
  \] 
  Here we have used the leg numbering for \(\Cst\)\nb-algebras: \(\Comp(\Hils)_{1}= 
  \Comp(\Hils)\otimes 1_{\Hils[L]\otimes\Hils\otimes\Hils[L]}\), 
  \(\Comp(\Hils[L])_{2}= 1_{\Hils}\otimes \Comp(\Hils[L])\otimes 1_{\Hils\otimes\Hils[L]}\), 
  \(\pi(A)_{3}=1_{\Hils\otimes\Hils[L]}\otimes\pi(A)\otimes 1_{\Hils[L]}\), 
  and~\(B_{4}=1_{\Hils\otimes\Hils[L]\otimes\Hils}\otimes B\).
  
  Using \(\Comp(\Hils)_{1}\Comp(\Hils[L])_{2}\pi(A)_{3}\Multunit_{13}\Corep{U}_{23}
  =\Comp(\Hils)_{1}\Comp(\Hils[L])_{2}\pi(A)_{3}\) we simplify the 
  last equation 
  \[
    \DuCorep{V}^{*}_{34}B_{4}\DuCorep{V}_{34}\Comp(\Hils)_{1}
   \Comp(\Hils[L])_{2}\pi(A)_{3}
   \DuCorep{V}_{34}^{*}\BrMultunit_{24}\DuCorep{V}_{34}
   =\Comp(\Hils)_{1}\Comp(\Hils[L])_{2} \pi(A)_{3}\DuCorep{V}^{*}_{34}B_{4}\DuCorep{V}_{34}.
  \] 
  Now multiplying~\(\Comp(\Hils)\) to the third leg from the left 
  and using~\eqref{eq:Poldes_Ducoat} we obtain 
  \[
    \Comp(\Hils)_{1}\Comp(\Hils[L])_{2}\Comp(\Hils)_{3}\pi(A)_{3}B_{4}
    \DuCorep{V}_{34}^{*}\BrMultunit_{24}\DuCorep{V}_{34}
  =\Comp(\Hils)_{1}\Comp(\Hils[L])_{2} \Comp(\Hils)_{3}\pi(A)_{3}B_{4}.
  \]
  Furthermore, the nondegeneracy of~\(\pi\) implies
  \[
    \Comp(\Hils)_{1}\Comp(\Hils[L])_{2}\Comp(\Hils)_{3}B_{4}
    \DuCorep{V}_{34}^{*}\BrMultunit_{24}\DuCorep{V}_{34}
  =\Comp(\Hils)_{1}\Comp(\Hils[L])_{2}\Comp(\Hils)_{3}B_{4}.  
  \] 
  Observe that~\eqref{eq:F_V-invariant_1} is equivalent to 
  \(\DuCorep{V}^{*}_{34}\BrMultunit_{24}\DuCorep{V}_{34}
  =\Corep{V}_{23}\BrMultunit_{24}\Corep{V}_{23}^{*}\). This 
  implies 
  \[
    \Comp(\Hils)_{1}\Comp(\Hils[L])_{2}\Comp(\Hils)_{3}B_{4}
    \Corep{V}_{23}\BrMultunit_{24}\Corep{V}_{23}^{*}
  =\Comp(\Hils)_{1}\Comp(\Hils[L])_{2}\Comp(\Hils)_{3}B_{4},
  \] 
  and it is  equivalent to 
  \[
  \Comp(\Hils)_{1}\Comp(\Hils[L])_{2}\Comp(\Hils)_{3}B_{4}
    \Corep{V}_{23}\BrMultunit_{24}
  =\Comp(\Hils)_{1}\Comp(\Hils[L])_{2}\Comp(\Hils)_{3}B_{4}\Corep{V}_{23}.
  \]
  Now \(\Corep{V}_{23}\) commutes with~\(B_{4}\). Moreover, 
  \(\Comp(\Hils[L])_{2}\Comp(\Hils)_{3}\Corep{V}_{23}
  =\Comp(\Hils[L])_{2} \Comp(\Hils)_{3}\). 
    Subsequently, the last equation becomes 
  \[
    \Comp(\Hils)_{1}\Comp(\Hils[L])_{2}\Comp(\Hils)_{3}B_{4}
    \BrMultunit_{24}
  =\Comp(\Hils)_{1}\Comp(\Hils[L])_{2}\Comp(\Hils)_{3}B_{4}.    
  \]
  Taking the slices on the first and third legs by~\(\omega,\omega'\in\Bound(\Hils)_{*}\) 
  give \((\Comp(\Hils[L])\otimes B)\BrMultunit =\Comp(\Hils[L])\otimes B\). 
  This shows that \(\BrMultunit\) is a unitary right multiplier of 
  \(\Comp(\Hils[L])\otimes B\). Multiplying both sides of the above equation 
  by~\(\BrMultunit^{*}\) from the right gives
  \(\Comp(\Hils[L])\otimes B =(\Comp(\Hils[L])\otimes B)\BrMultunit^{*}\); 
  hence \(\BrMultunit\) is also a unitary left multiplier of~\(\Comp(\Hils[L])\otimes B\). 
           
  \textbf{Part (4):}~
%
  Using the definition \(\Comult[B]\) 
  and the braided pentagon equation~\eqref{eq:top-braided_pentagon}
  we verify~\eqref{eq:brd_chr_leg2}:
  \[
  (\Id_{\Hils[L]}\otimes\Comult[B])\BrMultunit
  =\BrMultunit_{23}\BrMultunit_{12}\BrMultunit_{23}^*
  =\BrMultunit_{12}\Braiding{\Hils[L]}{\Hils[L]}_{23}
  \BrMultunit_{12}\Dualbraiding{\Hils[L]}{\Hils[L]}_{23}.
  \]
  Since~\(\BrMultunit\in\U(\Comp(\Hils[L])\otimes B)\), the right 
  hand side of the last equation is in~\(\U(\Comp(\Hils[L])\otimes 
  B\boxtimes B)\). Hence, the image of \(\Comult[B]\) lies in 
  \(\Mult(B\boxtimes B)\). Furthermore, taking slices on the first leg of 
  the first equality gives~\(\Comult[B](b)=\BrMultunit(b\otimes 1_{\Hils[L]})\BrMultunit^*\)
  for all~\(b\in B\).  This shows that~\(\Comult[B]\) is the unique \Star{}homomorphism 
  satisfying~\eqref{eq:brd_chr_leg2}.

  Next we recall~\((\Comp(\Hils[L])\otimes B)\BrMultunit =\Comp(\Hils[L])\otimes B\) and 
  use it in the following computation
  \begin{align*}
    (\Comp(\Hils[L])\otimes j_1(B))(\Id_{\Hils[L]}\otimes\Comult[B])\BrMultunit
    &=(\Comp(\Hils[L])\otimes j_1(B))\Bigl((\Id_{\Hils[L]}\otimes j_1)\BrMultunit (\Id_{\Hils[L]}\otimes j_2)\BrMultunit\Bigr) \\
    &= \Bigl((\Id_{\Hils[L]}\otimes j_1)\bigl((\Comp(\Hils[L])\otimes B)\BrMultunit\bigr)\Bigr)(\Id_{\Hils[L]}\otimes j_2)\BrMultunit \\
    &= (\Comp(\Hils[L])\otimes j_1(B))(\Id_{\Hils[L]}\otimes j_2)\BrMultunit.
  \end{align*}
  
  Slicing the first leg by \(\omega\in\Bound(\Hils[L])_*\) on both
  sides give \(j_1(B) \Comult[B](B) = j_1(B) j_2(B) = B\boxtimes
  B\).  A similar computation yields that \(\Comult[B](B)j_2(B) =
  B\boxtimes B\). Consequently, 
  \[
  \Comult[B](B)j_{2}(B)j_{1}(B) 
  =(B\boxtimes B)j_{1}(B)=j_{2}(B)j_{1}(B)j_{1}(B)=B\boxtimes B
  \] 
  shows that~\(\Comult[B]\) is nondegenerate.

  Once again, the braided pentagon equation~\eqref{eq:top-braided_pentagon} yields 
  coassociativity of~\(\Comult[B]\):
  \begin{align*}
    (\Comult[B]\boxtimes\Id_B)\Comult[B](b)
    =\BrMultunit_{12}\Braiding{\Hils[L]}{\Hils[L]}_{23}\Comult[B](b)_{12}
    \Dualbraiding{\Hils[L]}{\Hils[L]}_{23}\BrMultunit_{12}^*
    &= \BrMultunit_{12}\Braiding{\Hils[L]}{\Hils[L]}_{23}\BrMultunit_{12}
    b_1
    \BrMultunit_{12}^*\Dualbraiding{\Hils[L]}{\Hils[L]}_{23}\BrMultunit_{12}^*\\
    &=\BrMultunit_{23}\BrMultunit_{12}b_1\BrMultunit_{12}^*\BrMultunit_{23}\\
    &= (\Id_B\boxtimes\Comult[B])\Comult[B](b).
  \end{align*}

  Recall the diagonal action \(\beta\bowtie\beta\) of~\(\G\) on \(B\boxtimes B\)
  is described by~\eqref{eq:diag_G}:
  \[
  \beta\bowtie\beta\colon B\boxtimes B\to B\boxtimes B\otimes A,
  \qquad x\mapsto
  \corep{U}_{13}\corep{U}_{23}(x\otimes 1_A) \corep{U}^*_{23}
  \corep{U}^*_{13}.
  \]

  The invariance~\eqref{eq:F_U-invariant}
  of~\(\BrMultunit\) gives
  \begin{align*}
    \beta\bowtie\beta \Comult[B](b)
    &=\Corep{U}_{13}\Corep{U}_{23}\BrMultunit_{12}
    (b\otimes 1_{\Hils[L]\otimes\Hils})
    \BrMultunit^*_{12}\Corep{U}^*_{23}\Corep{U}^*_{13}\\
    &=\BrMultunit_{12}\Corep{U}_{13}\Corep{U}_{23}
    (b\otimes 1_{\Hils[L]\otimes\Hils})
    \Corep{U}^*_{23}\Corep{U}^*_{13}\BrMultunit^*_{12}
    = (\Comult[B]\otimes\Id_A)\beta(b);
  \end{align*}
  hence~\(\Comult[B]\) is \(\G\)\nb-equivariant.  Similarly, we can
  show that~\(\Comult[B]\) is \(\DuG\)\nb-equivariant.
 \end{proof}
 
\begin{definition}
 \label{def:br_Qntgrp}
 A \emph{braided \(\Cst\)\nb-quantum group} over~\(\G\) is a pair 
 \((B,\Comult[B])\) consisting of an object \((B,\beta,\hat{\beta})\) 
 and a morphism \(\Comult[B]\in\Mor(B,B\boxtimes B)\) in the 
 category~\(\YDcat(\G)\) constructed out of a manageable braided multiplicative 
 unitary~\(\BrMultunit\) over~\(\G\) described in the way as in 
 Theorem~\ref{the:braid_Qgrp}. Then we say~\(\Bialg{B}\) is \emph{generated} 
 by~\(\BrMultunit\).
\end{definition}  
Two braided \(\Cst\)\nb-quantum groups~\(\Bialg{B}\) and~\(\Bialg{B'}\) over~\(\G\) 
are \emph{isomorphic} if there is an isomorphism~\(f\in\Mor(B,B')\) in the category \(\YDcat(\G)\) such that~\((f\boxtimes f)\circ\Comult[B]=\Comult[B']\circ f\).
  
\subsection{Duals of braided $\textup{C}^{*}$-quantum groups}
Now we construct the reduced dual of~\(\Bialg{B}\) as a braided \(\Cst\)\nb-quantum 
group and prove the Pontrjagin duality theorem for braided quantum groups.

By~\cite{MRW2017}*{Definition 3.3, Proposition 3.4 \& 3.6}, 
the dual of a manageable braided multiplicative unitary~\(\BrMultunit\in\U(\Hils[L]\otimes\Hils[L])\) 
over \(\G\) with respect to~\((\corep{U},\corep{V})\), defined by 
\(\DuBrMultunit\defeq\Dualbraiding{\Hils[L]}{\Hils[L]}\BrMultunit^{*} 
\Braiding{\Hils[L]}{\Hils[L]}\in\U(\Hils[L]\otimes\Hils[L])\) is again 
a manageable braided multiplicative unitary over~\(\DuG\) with respect 
to~\((\corep{V},\corep{U})\). Also, \(\DuG\) is regular if and only if~\(\G\) 
is regular.

\begin{corollary}
 \label{cor:brd_dual}
 \(\DuBrMultunit\) generates a braided \(\Cst\)\nb-quantum group~\(\DuBialg{B}\) over \(\DuG\).
\end{corollary}
More precisely, 
\[
  \hat{B}\defeq \{(\omega\otimes\Id_{\Hils[L]})\DuBrMultunit\mid \omega\in\Bound(\Hils[L])_{*}\}^\CLS, 
 \quad 
 \DuComult[B](\hat{b})\defeq \DuBrMultunit (\hat{b}\otimes 1_{\Hils[L]}) 
 \DuBrMultunit^{*} 
 \quad
 \text{for all~\(\hat{b}\in\hat{B}\).}
\]
By construction \(\hat{B}\) is a~\(\DuG\)\nb-Yetter\nb-Drinfeld 
\(\Cst\)\nb-algebra with respect to the actions 
\(\hat{\delta}\colon\hat{B}\to\hat{B}\otimes\hat{A}\) 
 and~\(\delta\colon\hat{B}\to\hat{B}\otimes A\) defined 
 by
\begin{equation}
  \label{eq:YD_hat-B}
   \hat{\delta}(\hat{b})\defeq\corep{V}(\hat{b}\otimes 1_{\hat{A}})\corep{V}^{*}, 
   \qquad
   \delta(\hat{b})\defeq \corep{U}(\hat{b}\otimes 1_{A})\corep{U}^{*},
   \qquad\text{for all~\(\hat{b}\in\hat{B}\).} 
\end{equation}
In particular, a variant of Theorem~\ref{the:Yetter-Drinfeld_cat} shows that the 
monoidal product~\(B\widehat{\boxtimes}B\) in~\(\YDcat(\DuG)\) is defined by 
\(\hat{B}\widehat{\boxtimes}\hat{B}\defeq\iota_{1}(\hat{B})\iota_{2}(\hat{B})
\subset\Bound(\Hils[L]\otimes\Hils[L])\) where~\(\iota_{1},\iota_{2}\) are faithful representations 
of~\(\hat{B}\) on~\(\Hils[L]\otimes\Hils[L]\)
defined by~\(\iota_{1}(\hat{b})\defeq \hat{b}\otimes 1_{\Hils[L]}\) and \(\iota_{2}(\hat{b})\defeq 
\Dualbraiding{\Hils[L]}{\Hils[L]}(\hat{b}\otimes 1_{\Hils[L]})\Braiding{\Hils[L]}{\Hils[L]}\) for all \(\hat{b}\in\hat{B}\). Consequently, \(\DuComult[B]\colon \hat{B}\to \hat{B}\widehat{\boxtimes}\hat{B}\) is an arrow in~\(\YDcat(\DuG)\).
\begin{definition}
 The braided \(\Cst\)\nb-quantum group \(\Bialg{B}\) is said to be the 
 \emph{\textup{(}reduced\textup{)} dual} of~\(\Bialg{B}\).
\end{definition}
Once again,~\cite{MRW2017}*{Definition 3.3, Proposition 3.4 \& 3.6} imply the 
dual of~\(\DuBrMultunit\) is~\(\BrMultunit\). Consequently, we obtain the braided analogue of the \emph{Pontrjagin duality} theorem:
\begin{corollary}
 \label{cor:pontr_dual}
The  dual of~\(\DuBialg{B}\) is \textup{(}canonically\textup{)} isomorphic to \(\Bialg{B}\) as a braided 
\(\Cst\)\nb-quantum group over~\(\G\). 
\end{corollary}

\subsection{The bosonization}
 \label{sec:boson}
The reconstruction of the ordinary \(\Cst\)\nb-quantum group~\(\Qgrp{H}{C}\) and a projection with image~\(\G\) starting from a braided \(\Cst\)\nb-quantum group \(\Bialg{B}\) over~\(\G\) is called as \emph{bosonization}. In the compact case, that is, when~\(A\) and~\(B\) are unital, this has been already done in~\cite{MRW2016}*{Theorem 6.4}. We extend this result for general \(\Cst\)\nb-quantum groups, essentially, using the same ingredients.
According to Theorem~\ref{the:Yetter-Drinfeld_cat} \(\YDcat(\G)\) is a monoidal category and \((B,\beta,\hat{\beta})\) is an object of the category \(\YDcat(\G)\). Also, regularity of~\(\G\) makes 
\(A\) an object of~\(\YDcat(\G)\) as well, see Example~\ref{ex:yett-drinf_reg}. Then 
\(A\boxtimes B\defeq (A\otimes 1_{\Hils[L]})\DuCorep{V}^{*} (1_{\Hils}\otimes B)\DuCorep{V}\) as shown in~\cite{MRW2016}*{Page 19}. Here we have suppressed the faithful representations of~\(A\) and \(B\) on~\(\Hils\) and~\(\Hils[L]\), respectively. 
In fact, \(B\ni b\mapsto \DuCorep{V}^{*} (1_{\Hils}\otimes B)\DuCorep{V} 
\in \hat{A}\otimes B\) defines a left action of the co-opposite quantum group  
\(\DuG^{\textup{cop}}:=(\hat{A}, \flip\circ\DuComult[A])\) of~\(\DuG\) and \(A\boxtimes B =\DuG^{\textup{cop}}\ltimes B \textup{ (}\cong B\rtimes_{\hat{\beta}}\DuG\textup{)}\). 

By virtue of~\cite{MRW2016}*{Proposition 6.3} we get 
an injective morphism 
\(\Psi\colon A\boxtimes B\boxtimes B\to A\boxtimes B\otimes A\boxtimes B\)
defined by
\begin{equation}
 \label{eq:Psi_twised}
 A\boxtimes B\boxtimes B\ni x\mapsto \Multunit_{12}\Corep{U}_{23}\DuCorep{V}_{34}^{*}x_{124}\DuCorep{V}_{34}
 \Corep{U}_{23}^{*}\Multunit[*]_{12}.
\end{equation}

\begin{proposition}
 \label{prop:big_qntgrp}
 Let~\(C=A\boxtimes B\). Define~\(\Comult[C]\in\Mor(C, C\otimes C )\) 
 by~\(\Comult[C]\defeq\Psi\circ(\Id_{B}\boxtimes\Comult[B])\). Then 
 \(\Qgrp{H}{C}\) is the \(\Cst\)\nb-quantum group generated by \(\Multunit[C]\) in~\textup{\eqref{eq:big_MU}}. Moreover, \(\G[H]\) is a \(\Cst\)\nb-quantum group with a projection and~\(\G\) becomes image of the projection.
\end{proposition}
\begin{proof}
 Let~\(L=\{(\omega\otimes\omega'\otimes\Id_{\Hils\otimes\Hils[L]})\Multunit[C] 
   \mid\text{ \(\omega\in\Bound(\Hils)_{*}\), \(\omega'\in\Bound(\Hils[L])_{*}\)}\}^\CLS\). 
 
 Using~\eqref{eq:first_leg_slice} we get
 \begin{align*}
   L &=\{(\omega\otimes\omega'\otimes\Id_{\Hils\otimes\Hils[L]})
   \Multunit_{13}\Corep{U}_{23}\DuCorep{V}^*_{34}\BrMultunit_{24}\DuCorep{V}_{34} 
   \mid\text{ \(\omega\in\Bound(\Hils)_{*}\), \(\omega'\in\Bound(\Hils[L])_{*}\)}\}^\CLS \\
   &=\{(\omega'\otimes\Id_{\Hils\otimes\Hils[L]})
   ((1\otimes a\otimes 1)\Corep{U}_{12}\DuCorep{V}^*_{23}\BrMultunit_{13}\DuCorep{V}_{23})
   \mid\text{\(\omega'\in\Bound(\Hils[L])_{*}\), \(a\in A\)}\}^\CLS
 \end{align*}
 For~\(\omega'\in\Bound(\Hils[L])_{*}\) and~\(\xi\in\Comp(\Hils[L])\) 
 define \(\omega'\cdot\xi\in\Bound(\Hils[L])_{*}\) by 
 \(\omega'\cdot\xi(y)\defeq \omega'(\xi y)\). 
 
 Replacing~\(\omega'\) by~\(\omega'\cdot \xi\) 
 in the last expression we get 
 \[
   L=\{(\omega'\otimes\Id_{\Hils\otimes\Hils[L]})
   (((\xi\otimes a)\Corep{U}\otimes 1_{\Hils[L]})
   \DuCorep{V}^*_{23}\BrMultunit_{13}\DuCorep{V}_{23})
   \mid\text{\(\omega'\in\Bound(\Hils[L])_{*}\), \(\xi\in\Comp(\Hils[L])\), \(a\in A\)}\}^\CLS
 \] 
 We may also replace \((\xi\otimes a)\Corep{U}\) by 
 \(\xi\otimes a\) for~\(\xi\in\Comp(\Hils[L])\), 
 \(a\in A\), because 
 \(\corep{U}\in\U(\Comp(\Hils[L])\otimes A)\) and 
 \(\Corep{U}=(\Id_{\Hils[L]}\otimes\pi)\corep{U}\). We have 
 \begin{align*}
   L &=\{(\omega'\otimes\Id_{\Hils\otimes\Hils[L]})
   ((\xi\otimes a\otimes 1_{\Hils[L]})
   \DuCorep{V}^*_{23}\BrMultunit_{13}\DuCorep{V}_{23} )
   \mid\text{\(\omega'\in\Bound(\Hils[L])_{*}\), \(\xi\in\Comp(\Hils[L])\), \(a\in A\)}\}^\CLS \\
  &= \{(\omega'\otimes\Id_{\Hils\otimes\Hils[L]})
   ((1_{\Hils[L]}\otimes a\otimes 1_{\Hils[L]})
   \DuCorep{V}^*_{23}\BrMultunit_{13}\DuCorep{V}_{23}) 
   \mid\text{\(\omega'\in\Bound(\Hils[L])_{*}\), \(a\in A\)}\}^\CLS .
 \end{align*}
 Finally using~\eqref{eq:slice_first_brmult} we obtain
 \begin{align*}
  L &=\{(\omega'\otimes\Id_{\Hils\otimes\Hils[L]})
   ((1\otimes a\otimes 1)\DuCorep{V}^*_{23}\BrMultunit_{13}\DuCorep{V}_{23} )
   \mid\text{\(\xi\in\Comp(\Hils[L])\), \(a\in A\), \(\omega'\in\Bound(\Hils[L])_{*}\)}\}^\CLS  \\
 &=(A\otimes 1_{\Hils[L]})\DuCorep{V}^{*}(1_{\Hils}\otimes B)\DuCorep{V}=C
 \end{align*}
 
 Now for any~\(c\in C=A\boxtimes B\subset\Bound(\Hils\otimes\Hils[L])\) 
 \begin{multline*}
   \Comult[C](c)=\Psi\bigl((\Id_{B}\boxtimes\Comult[B])(c)\bigr) 
   =\Psi(\BrMultunit_{23}(c\otimes 1_{\Hils[L]})\BrMultunit_{23}^{*})
   \\=\Multunit_{12}\Corep{U}_{23}\DuCorep{V}_{34}^{*}\BrMultunit_{24} 
     (c\otimes 1_{\Hils\otimes\Hils[L]})
     \BrMultunit_{24}^{*}\Ducorep{V}_{34}\Corep{U}_{23}^{*}\Multunit[*]_{12}
     =(\Multunit[C])(c\otimes 1)(\Multunit[C])^{*}.
 \end{multline*}
 Theorem~\ref{the:Cst_quantum_grp_and_mult_unit} shows that \(\Comult[C]\in\Mor(C, C\otimes C)\) is the unique element satisfying~\((\Id_{\hat{C}}\otimes\Comult[C])\Multunit=\Multunit_{12}\Multunit_{13}\). Thus 
 \(\Bialg{C}\) is the \(\Cst\)\nb-quantum group generated by~\(\Multunit[C]\).
The unitary \(\ProjBichar\) in~\eqref{eq:Proj_Unit} is a projection on~\(\G[H]\) 
with image \(\Qgrp{G}{A}\), see Lemma~\ref{lemm:rep_hat-A}.
\end{proof}
Suppose~\(\multunit[C]\in\U(\hat{C}\otimes C)\) is the reduced bicharcater of~\(\Qgrp{H}{C}\), \(\projbichar\in\U(\hat{C}\otimes C)\) is the projection on~\(\G[H]\), and the image of~\(\projbichar\) is the regular~\(\Cst\)\nb-quantum group~\(\Qgrp{G}{A}\). Then we can construct a manageable braided multiplicative unitary~\(\BrMultunit\) over~\(\G\) described in the way as in~\cite{MRW2017}*{Theorem 3.9}. Suppose, \(\BrMultunit\) gives rise to the braided \(\Cst\)\nb-quantum group~\(\Bialg{B}\) over~\(\G\) and~\(\G[H]_{1}=\Bialg{C_{1}}\) is the associated  bosonization with projection~\(\projbichar_{1}\in\U(\hat{C_{1}}\otimes C_{1})\). As a consequence of \cite{MRW2017}*{Theorem 3.10} and~\cite{MR2019}*{Theorem 2.18}, there is a Hopf~\Star{}isomorphism~\(f\in \Mor(C,C_{1})\) such that~\(\projbichar_{1}=(\hat{f}\otimes f)\projbichar\). Hence,~\((\G[H],\projbichar)\) is isomorphic to~\((\G[H]_{1},\projbichar_{1})\).

Hence, starting with a~\(\Cst\)\nb-quantum group~\(\G[H]\) with a projection~\(\projbichar\) whose image is a regular~\(\Cst\)\nb-quantum group~\(\G\) we can construct a braided \(\Cst\)\nb-quantum group~\(\Bialg{B}\) over~\(\G\) and reconstruct~\(\G[H]\) as the bosonization of~\(\Bialg{B}\) and the projection~\(\projbichar\) on~\(\G[H]\), up to isomorphism. 

Next we show that the construction~\((\G[H],\projbichar)\to\Bialg{B}\) respects the isomorphisms. For that matter, let us recall the Drinfeld's double~\(\mathfrak{D}(\G[H]_{i})=\Bialg{\mathcal{D}_{i}}\) of~\(\G[H]_{i}\) from~\cite{R2015a}*{Example 5.12} for~\(i=1,2\). The embeddings~\(\rho_{i},\theta_{i}\colon C_{i},\hat{C_{i}}\rightrightarrows \mathcal{D}_{i}\) are Hopf~\Star{}homomorphisms. Consider the faithful representation~\(\pi_{i}\in\Mor(\mathcal{D}_{i},\Comp(\Hils_{i}))\) for~\(i=1,2\). Define~\(\corep{U}_{i}\defeq(\pi_{i}\circ\theta_{i}\otimes\Id_{C})\projbichar_{i}\in\U(\Comp(\Hils_{i})\otimes C_{i})\), \(\corep{V}_{i}\defeq(\pi_{i}\circ\rho_{i}\otimes\Id_{\hat{C_{i}}})\widehat{\projbichar}\in\U(\Comp(\Hils_{i})\otimes\hat{C})\) and~\(\BrMultunit_{i}\defeq (\pi_{i}\circ\theta_{i}\otimes\pi_{i}\circ\rho_{i})\projbichar_{i}^{*}\multunit[C_{i}]\in\U(\Hils_{i}\otimes\Hils_{i})\) for~\(i=1,2\). Then \cite{MRW2017}*{Theorem 3.10} says that \(\BrMultunit_{i}\) is a manageable braided multiplicative unitary over~\(\G_{i}\) relative to~\((\corep{U}_{i},\corep{V}_{i})\) for~\(i=1,2\). Let~\(\G_{i}=\Bialg{A_{i}}\) be the image of~\(\projbichar_{i}\) and for~\(i=1,2\). Assume~\(\G_{1}\) and~\(\G_{2}\) are regular~\(\Cst\)\nb-quantum groups. Then we construct 
the braided \(\Cst\)\nb-quantum group~\(\Bialg{B_{i}}\) from \(\BrMultunit_{i}\) for~\(i=1,2\) in the way as in Theorem~\ref{the:braid_Qgrp}.

Suppose~\(f\in\Mor(C_{1},C_{2})\) defines an isomorphism between~\((\G[H]_{1},\projbichar_{1})\) and~\((\G[H]_{2},\projbichar_{2})\). Consider the dual Hopf~\Star{}isomorphism~\(\hat{f}\in\Mor(\hat{C}_{1},\hat{C}_{2})\). 
Then~\((\hat{f}\otimes f)\projbichar_{1}=\projbichar_{2}\) by \eqref{eq:projiso-rel}. Then \(f\) induces a Hopf \Star{}isomorphism \(h\colon \pi_{1}(D_{1})\cong D_{1}\to D_{2}\cong\pi_{2}(D_{2})\) such that~\(h\circ \pi_{1}\circ\rho_{1}=\pi_{2}\circ\rho_{2}\circ f\) and~\(h\circ\pi_{1}\circ \theta_{1}=\pi_{2}\circ \theta_{2}\circ \hat{f}\). So,~\((h\otimes f)\corep{U}_{1}=\corep{U}_{2}\), \((h\otimes\hat{f})\corep{V}_{1}=\corep{V}_{2}\) and~\((h\otimes h)\BrMultunit_{1}=\BrMultunit_{2}\). Then~\(B_{i}=\{(\omega\otimes\Id)\BrMultunit_{i}\mid \omega\in\Bound(\Hils_i)_{*}\}^\CLS\) for~\(i=1,2\). Also, the Landstad\nb-Vaes algebras in Proposition~\ref{prop:Landstad_slices} associated to~\((\G[H]_{1},\projbichar_{1})\) and~\((\G[H]_{2},\projbichar_{2})\), respectively are isomorphic. Consequently, the restriction~\(h_{B}\) of~\(h\) on~\(B_{1}\) defines an isomorphism between \(B_{1}\) and \(B_{2}\). Since, \(\G_{1}\) and~\(\G_{2}\) are isomorphic \(\Cst\)\nb-quantum groups, we may identify~\(A_{2}\) with~\(f(A_{1})\). Let~\(\G=\G_{1}=\G_{2}\). Now, the~\(\G\)\nb-action on~\(B_{i}\) is given by~\(B_{i}\ni b_{i}\to\corep{U}_{i}(b_{i}\otimes 1_{A})\corep{U}_{i}^{*}\) for~\(i=1,2\). Then~\(h_{B}\) is~\(\G\)\nb-equivariant. Similarly, the~\(\DuG\)\nb-action on~\(B_{i}\) is implemented by~\(\corep{V}_{i}\) for~\(i=1,2\); hence~\(h_{B}\) is also~\(\DuG\)\nb-equivariant. Therefore, \(h_{B}\in\Mor(B_{1},B_{2})\) is an isomorphism in the category~\(\YDcat(\G)\). Denote the embeddings~\(j_{1,i},j_{2,i}\colon B_{i}\rightrightarrows 
B_{i}\boxtimes B_{i}\) for~\(i=1,2\). Following~\eqref{eq:brd_chr_leg2} we characterise~\(\Comult[B_{1}]\) and~\(\Comult[B_{2}]\) as follows
\begin{align*}
  (\Id_{\Hils_{1}}\otimes\Comult[B_{1}])\BrMultunit_{1} 
  &= (\Id_{\Hils_{1}}\otimes j_{1,1})\BrMultunit_{1} (\Id_{\Hils_{1}}\otimes j_{2,1})\BrMultunit_{1}, \\
  (\Id_{\Hils_{2}}\otimes\Comult[B_{2}])\BrMultunit_{2} 
  &= (\Id_{\Hils_{2}}\otimes j_{1,2})\BrMultunit_{2} (\Id_{\Hils_{2}}\otimes j_{2,2})\BrMultunit_{2}.
\end{align*}
Then~\(h\otimes (h|_{B}\boxtimes h|_{B})\) maps the first equation to the second equation; hence \((h|_{B}\boxtimes h|_{B})\circ\Comult[B_{1}]=\Comult[B_{2}]\circ h_{B}\). Hence, \(h_{B}\) defines an isomorphism of braided~\(\Cst\)\nb-quantum groups between~\(\Bialg{B_{1}}\) and~\(\Bialg{B_{2}}\).

On the other hand, let~\(\Bialg{B}\) be a braided~\(\Cst\)\nb-quantum group over a regular~\(\Cst\)\nb-quantum group~\(\G\). Suppose~\(\Qgrp{H}{C}\) is the bosonization of~\(\Bialg{B}\) and \(\projbichar\) is the projection on~\(\G[H]\) as constructed in the way as in Proposition~\ref{prop:big_qntgrp}. As before, we construct a braided \(\Cst\)\nb-quantum group~\(\Bialg{B_{1}}\) be the over~\(\G\) from~\((\G[H],\projbichar)\) and its bosonization~\(\G[H]_{1}=\Bialg{C_{1}}\) along with the projection~\(\projbichar_{1}\) on~\(\G[H]_{1}\). Then~\((\G[H],\projbichar)\) and~\((\G[H]_{1},\projbichar_{1})\) are isomorphic \(\Cst\)\nb-quantum groups with projection. Consequently, \(\Bialg{B}\) and~\(\Bialg{B_{1}}\) are isomorphic braided~\(\Cst\)\nb-quantum groups.

Finally, we are going to show that the construction \(\Bialg{B}\to (\G[H],\projbichar)\) respects the isomorphisms. Suppose~\(\Bialg{B_{1}}\) and~\(\Bialg{B_{1}}\) are isomorphic braided~\(\Cst\)\nb-quantum groups over a regular \(\Cst\)\nb-quantum group~\(\Qgrp{G}{A}\). Let~\(f\in\Mor(B_{1},B_{2})\) be the isomorphism in the category~\(\YDcat(\G)\). This extends to an isomorphism~\(h\defeq\Id_{A}\boxtimes f\in\Mor(A\boxtimes B_{1}, A\boxtimes B_{2})\) such that~\(h\circ i_{A}^{1}=i_{A}^{2}\) and~\(h\circ i_{B_{1}}=i_{B_{2}}\circ f\). Here \(i_{A}^{k},i_{B_{k}}\colon A,B_{k}\rightrightarrows C_{k}=A\boxtimes B_{k}\) are the canonical morphisms for~\(k=1,2\). In order to keep track of the copy of~\(A\) inside~\(C_{k}\) we use different notations~\(i_{A}^{1}, i_{A}^{2}\) for their embeddings, whereas~\(i_{A}^{1}=i_{A}^{2}\).

Suppose~\(\G[H]_{k}=\Bialg{C_{k}}\) is the~\(\Cst\)\nb-quantum group with the projection~\(\projbichar_{k}\) given by Proposition~\ref{prop:big_qntgrp} for~\(k=1,2\). Then the images of~\(\projbichar_{1}\) and~\(\projbichar_{2}\) are isomorphic to the regular \(\Cst\)\nb-quantum group~\(\G\). Also, \(B_{k}\) is identified with the Landstad-Vaes algebra~\(i_{B_{k}}(B_{k})\subset\Mult(C)\) for the \(\G\)\nb-product~\((C_{k},\Delta^{k}_{L},i_{A}^{k})\) (\(\Delta^{k}_{L}\) is the left action of~\(\G\) on~\(C_{k}\)) induced by the projection~\(\projbichar_{k}\) on~\(\G[H]_{k}\) in Proposition~\ref{prop:Landstad_slices} for~\(k=1,2\). 

Recall the injective morphism~\(\Psi_{k}\colon A\boxtimes B_{k}\boxtimes B_{k}\to C_{k}\otimes C_{k}\) constructed in~\cite{MRW2016}*{Proposition 6.3} for~\(k=1,2\). On the embeddings~\(j_{1}^{k}, j_{2}^{k}, j_{3}^{k}\) they are defined by 
\begin{align*}
  \Psi_{k} j_{1}^{k}(a)
   =(i_{A}^{k}\otimes i_{A}^{k})\Comult[A](a), & \qquad
  \Psi_{k} j_{2}^{k}(b_{k})
  =(i_{B_{k}}\otimes i_{A}^{k})\beta^{k}(b_{k}),\\
  \Psi_{k} j_{3}^{k}(b_{k})
  =1_{C_{k}}\otimes i_{B_{k}}(b_{k}), 
  &\qquad\text{for~\(a\in A\), \(b_{k}\in B_{k}\) for~\(k=1,2\).}
\end{align*}
Here~\(\beta^{k}\in\Mor(B_{k},B_{k}\otimes A)\) is the~\(\G\)\nb-action on~\(B_{k}\) for~\(k=1,2\). Now~\(\Comult[C_{k}]=\Psi_{k}\circ(\Id_{A}\boxtimes\Comult[B_{k}])\) for~\(k=1,2\).

Clearly, 
\begin{align*}
  (h\otimes h)\circ\Comult[C_{1}]\circ i_{A}^{1}=(h\circ i^{1}_A\otimes h\circ i^{1}_{A})\circ\Comult[A]=(i^{2}_{A}\otimes i^{2}_{A})\circ\Comult[A]
&=\Comult[C_{2}]\circ i^{2}_{A}\\
&=\Comult[C_{2}]\circ h\circ i^{1}_{A}.
\end{align*}
So the restriction~\(h|_{A}\) is a Hopf~\Star{}isomorphism. 
Using the fact that~\(f\) is~\(\G\)\nb-equivariant and~\(h\circ i_{B_{1}}=i_{B_{2}}\circ f\) we verify \((h\otimes h)\circ \Psi|_{A, B_{1}}\circ j^{1}_{l}=\Psi_{A,B_{2}}\circ j^{2}_{l}\circ f\) for~\(l=2,3\). This yields
\begin{align*}
    (h\otimes h)\circ\Comult[C_1]\circ i_{B_{1}}
    &=(h\otimes h)\Psi_{A,B_{1}}(\Id_{A}\boxtimes\Comult[B_{1}])\circ i_{B_{1}}\\
    &=\Psi_{A,B_{2}}\circ (\Id_{A}\boxtimes (f\boxtimes f)\circ\Comult[B_{1}])\circ i_{B_{1}}\\
   &=\Psi_{A,B_{2}}\circ (\Id_{A}\boxtimes\Comult[B_{1}])\circ(\Id_{A}\boxtimes f)\circ i_{B_{1}}\\
   &=\Comult[C_{2}]\circ h\circ i_{B_{1}}
     =\Comult[C_{2}]\circ h\circ i_{B_{1}}.
\end{align*}
Therefore, \(h\) defines an isomorphism between~\(\G[H]_{1}\) and~\(\G[H]_{2}\). Hence, \((\G[H]_{1},\projbichar_{1})\) and~\((\G[H]_{2},\projbichar_{2})\) are isomorphic~\(\Cst\)\nb-quantum groups with projection. Summarising, we have the following result.
\begin{theorem}
 \label{the:brdq_qgpp}
 Isomorphism classes of braided \(\Cst\)\nb-quantum groups over a regular \(\Cst\)\nb-quantum group~\(\G\) are in one to one correspondence with the isomorphism classes of \(\Cst\)\nb-quantum groups with a projection generating~\(\G\) as its image.
\end{theorem}

\section{Examples}
 \label{sec:Examples}
\subsection{C*-quantum groups with an idempotent Hopf~$^*$-homomorphism}

Suppose,~\(f\in\Mor(C,C)\) is an idempotent Hopf~\Star{}homomorphism on a \(\Cst\)\nb-quantum group~\(\Qgrp{H}{C}\). Let~\(\multunit[C]\in\U(\hat{C}\otimes C)\) be the reduced bicharacter of~\(\G[H]\). The unitary \(\projbichar\defeq (\Id_{\hat{C}}\otimes f)\multunit[C]\in\U(\hat{C}\otimes C)\) is the unique bicharacter corresponding to~\(f\). Since~\(f\) is idempotent, by~\cite{MRW2012}*{Definition 3.5} 
\(\projbichar\) also satisfies~\eqref{eq:proj_cond}; hence~\(\projbichar\) is a projection on~\(\G[H]\). 

Clearly, \(A\defeq\textup{Im}(f)=\{(\omega\otimes\Id_{C})\projbichar 
\mid \omega\in\hat{C}^{*}\}^\CLS\) and~\(\Comult[A]\defeq\Comult[C]|_{A}\in\Mor(A,A\otimes A)\) satisfy~\((\Id_{\hat{C}}\otimes\Comult[A])\projbichar 
=\projbichar_{12}\projbichar_{13}\). So, \(\Qgrp{H}{C}\) is a quantum group with 
projection~\(\projbichar\) with image \(\Qgrp{G}{A}\). Theorem~\ref{the:brdq_qgpp} says that there exists a unique braided \(\Cst\)\nb-quantum group~\(\Bialg{B}\) over~\(\G\) and~\(\G[H]\) is the associated bosonization. 

Quantum~\(\textup{E}(2)\) groups~\cite{W1992a}, quantum~\(az+b\) groups~\cites{W2001, S2005} and quantum~\(ax+b\) groups \cite{WZ2002} are examples of \(\Cst\)\nb-quantum groups with an idempotent Hopf $^*$-homomorphism generating the multiplicative subgroups~\(\T, q^{\Z+i\R}\) (for a suitably chosen deformation parameter \(q\in\C\setminus\{0\}\)), and \(\R_{>0}^{\times}\) of~\(\C\setminus\{0\}\) as their images, respectively. For more details we refer~\cite{R2013}*{Section 6.2.1}, \cite{MRW2016}*{Section 4} and~\cite{KS2015}*{Example 3.7}. 

\subsection{Braided compact quantum groups}
 Suppose~\(\Qgrp{G}{A}\) is a compact quantum group. 
 By~\cite{MRW2016}*{Definition 6.1}, 
 a \emph{braided compact quantum group} over~\(\G\) 
 a pair~\(\Bialg{B}\) consisting 
 of a unital~\(\G\)\nb-Yetter-Drinfeld \(\Cst\)\nb-algebra 
 \((B,\beta,\hat{\beta})\) and a unital \Star{} 
 homomorphism~\(\Comult[B]\colon B\to B\boxtimes B\) satisfy~\eqref{eq:br_coasso} and~\eqref{eq:br_cancel}. 

\begin{proposition}
  Every braided compact quantum group over~\(\G\) 
  is a braided \(\Cst\)\nb-quantum group with 
  the underlying \(\Cst\)\nb-algebra being unital.
\end{proposition}
\begin{proof}
Let~\(\Bialg{B}\) be a braided compact quantum group over~\(\G\). 
Suppose, \(\Qgrp{H}{C}\) is the bosonization, which is a compact quantum group, of~\(\Bialg{B}\) as in~\cite{MRW2016}*{Theorem 6.4}. Let~\(h\) be the Haar state 
of~\(\G[H]\) and let~\(\Hils_{h}\) be the GNS space. Then the right regular 
representation~\(\Multunit[C]\in\U(\Hils_{h}\otimes\Hils_{h})\) of~\(\G[H]\) 
on~\(\Hils_{h}\) is a manageable multiplicative unitary and generates~\(\Qgrp{H}{C}\). 

Moreover, there is a projection on~\(\G[H]\) consisting of the canonical embedding 
 \(i_{A}\colon A\hookrightarrow A\boxtimes B=C\) and the left quantum group 
 homomorphism \(\Delta_{L}\in\Mor(C,A\otimes C)\) given by~\cite{MRW2017}*{Proposition 2.8 \& 2.10}. Let~\(\projbichar\in
 \U(\hat{C}\otimes C)\) be the projection equivalent to~\((i_{A},\Delta_{L})\). 
Following \cite{MRW2017}*{Theorem 3.9} we may construct a manageable 
 braided multiplicative unitary~\(\BrMultunit\in\U(\overline{\Hils_{h}}\otimes\Hils_{h}  \otimes\overline{\Hils_{h}}\otimes\Hils_{h})\) over~\(\G\). Subsequently, the braided 
 \(\Cst\)\nb-quantum group generated by~\(\BrMultunit\) is isomorphic to 
 \(\Bialg{B}\).
\end{proof}
 
\subsection{The complex quantum planes and their bosonizations}
\label{sec:qnt_E_2}
Throughout this section, we shall consider \(\Qgrp{G}{\Cont(\T)}\) and 
\(\DuG=\Bialg{\Contvin(\Z)}\). Let~\(\Hils=\ell^{2}(\Z)\) and let~\(\{e_{p}\}_{p\in\Z}\) be 
an orthonormal basis of \(\Hils\). A multiplicative unitary~\(\Multunit\in\U(\Hils\otimes\Hils)\) generating~\(\T\) is given by~\(\Multunit(e_{k}\otimes e_{l})\defeq e_{k}\otimes e_{l+k}\) for all~\(k,l\in\Z\). 

Since~\(\T\) and~\(\Z\) are Abelian groups, the quantum codouble of~\(\G\) coincides 
with \(\Z\times\T\), while viewed as \(\Cst\)\nb-quantum group. Similarly, the category of \(\G\)\nb-Yetter\nb-Drinfeld and the category of \(\DuG\)\nb-Yetter\nb-Drinfeld \(\Cst\)\nb-algebras are equivalent to the category of~\(\Z\times\T\)\nb-\(\Cst\)\nb-algebras 
and~\(\T\times\Z\)\nb-\(\Cst\)\nb-algebras, respectively. 

Fix~\(\Hils[L]=\Hils\otimes\Hils\) and the orthonormal basis~\(\{e_{i,j}\defeq e_{i}\otimes e_{j}\}_{i,j\in\Z}\). The canonical representations of~\(\Cont(\T)\cong \Cst(\Z)\) and~\(\Contvin(\Z)\cong\Cst(\T)\) on \(\Hils[L]\) through the the unitary~\(\udrinf\) and the self adjoint operator \(\Nhatdrinf\) with spectrum~\(\Z\) and commuting with~\(\udrinf\). 
Subsequently, the right and left representations~\(\Corep{U}\in\U(\Hils[L]\otimes\Hils)\) and \(\DuCorep{V}\defeq\Flip \Corep{V}^{*}\Flip\in\U(\Hils\otimes\Hils[L])\) and the resulting braiding operator~\(\Braiding{\Hils[L]}{\Hils[L]}\) are defined by
\begin{equation}
 \label{eq:Mar0619}
\Corep{U}=\Multunit_{23}, 
\quad \DuCorep{V}=\Multunit_{12},
\quad 
\Braiding{\Hils[L]}{\Hils[L]}=Z\Flip=\Multunit[*]_{23}\Flip.
\end{equation}

For a fixed~\(0<q<1\), let \(\C_{q}^{\times}\) be the subgroup~\(q^{\Z+i\R}\) 
of the multiplicative group~\(\C\setminus\{0\}\) and let~\(\conj{\C}_{q}=\C_{q}^{\times}\cup\{0\}\).  Define~\(\Upsilon=\Ph{\Upsilon}\Mod{\Upsilon}\) as a closed operator acting on~\(\Hils[L]\) by
\[
 \Ph{\Upsilon}e_{i,j}\defeq e_{i,j+1},
 \qquad
 \Mod{\Upsilon} e_{i,j}\defeq q^{2i+j}e_{i,j}, 
 \qquad 
 \Upsilon e_{i,j}\defeq q^{2i+j}e_{i,j+1}.
\]
 The operator~\(\Ph{\Upsilon}\) is unitary, \(\Mod{\Upsilon}\) is a strictly positive operator such that 
 \begin{equation}
  \label{eq:prop_of_Upsilon}
   \Ph{\Upsilon}\Mod{\Upsilon}\Ph{\Upsilon}^*=q^{-1}\Mod{\Upsilon},
   \qquad
   \textup{Sp}(\Mod{\Upsilon})= q^{\Z}\cup\{0\}.
\end{equation}
Thus~\(\Upsilon^{-1}e_{i,j}\defeq q^{-2i-j+1}e_{i,j-1}\) and the polar decomposition 
\(\Upsilon^{-1}=\Ph{\Upsilon^{-1}}\Mod{\Upsilon^{-1}}\) gives a unitary operator 
\(\Ph{\Upsilon^{-1}}\), a strictly positive operator \(\Mod{\Upsilon^{-1}}\) 
with spectrum~\(q^{\Z}\cup\{0\}\), and \(\Ph{\Upsilon^{-1}}\) and~\(\Mod{\Upsilon^{-1}}\) 
satisfy the following commutation relation
 \begin{equation}
  \label{eq:prop_of_Upsilon}
   \Ph{\Upsilon^{-1}}\Mod{\Upsilon^{-1}}\Ph{\Upsilon^{-1}}^*=q\Mod{\Upsilon^{-1}}.
\end{equation}
\begin{proposition}
 \label{prop:Cstar-q-plane}
 Define
\begin{equation}
  \label{eq:homogen_sp_E_2}
   B\defeq\left\{\sum_{k\in\mathbb{Z}}^{\text{finite}}\Ph{\Upsilon^{-1}}^{k}f_k(\Mod{\Upsilon^{-1}}) 
   \middle|\text{ } f_k \in\Contvin(\conj{\C}_{q}),\text{ }f_k(0)=0\text{ for } k\neq 0\right\}^{\CLS}.
\end{equation}
 Then \(B\) is a \(\Cst\)\nb-algebra, \(\Upsilon^{-1}\eta B\) and \(B\) is generated by~\(\Upsilon^{-1}\).
\end{proposition}
\begin{proof}
 For any two elements~\(\Ph{\Upsilon^{-1}}^{k}f_{k}(\Mod{\Upsilon^{-1}}), \Ph{\Upsilon^{-1}}^{l}g_{l}(\Mod{\Upsilon^{-1}})\in B\) we observe 
 that 
 \[
   \Ph{\Upsilon^{-1}}^{k}f_{k}(\Mod{\Upsilon^{-1}})\Ph{\Upsilon^{-1}}^{l}g_{l}(\Mod{\Upsilon^{-1}})
 =\Ph{\Upsilon}^{k-l}f_{k}(q^{l}\Mod{\Upsilon^{-1}})g_{l}(\Mod{\Upsilon^{-1}})\in B
 \]
and~\(B\) is \Star{}invariant; hence~\(B\) is a \(\Cst\)\nb-algebra. Rest of the proof 
follows using a similar line of argument used in~\cite{S2010}*{Proposition 4.1 (2-3)}.
\end{proof}
The maps~\(\gamma\colon \Upsilon^{-1}\to\Corep{U}(\Upsilon^{-1}\otimes 1)\Corep{U}^{*}=\Upsilon^{-1}\otimes u^{*}\eta B\otimes\Cont(\T)\) and \(\hat{\gamma}\colon\Upsilon^{-1}\to\Corep{V}(\Upsilon^{-1}\otimes 1)\Corep{V}^{*}=\Upsilon^{-1}\otimes q^{-2\hat{N}}\eta B\otimes\Contvin(\Z)\) define~\(\T\) and~\(\Z\) actions on~\(B\), respectively. Here~ \(u\) and~\(\hat{N}\) are the generators of~\(\Cont(\T)\) and~\(\Contvin(\Z)\) defined by~\(ue_{p}\defeq e_{p+1}\) and~\(\hat{N}e_{p}\defeq pe_{p}\), respectively. Thus~\(B\) is a~\(\Z\times\T\)\nb-\(\Cst\)\nb-algebra and using the braiding unitary~\(\Braiding{}{}\) in~\eqref{eq:Mar0619} we define~\(B\boxtimes B\). On the generator~\(\Upsilon^{-1}\) the canonical \(\Z\times\T\)\nb-equivariant 
embeddings~\(j_{1},j_{2}\in\Mor(B,B\boxtimes B)\) are defined by 
\begin{equation}
 \label{eq:Embed-BboxB}
j_{1}(\Upsilon^{-1})\defeq\Upsilon^{-1}\otimes 1, 
\qquad 
j_{2}(\Upsilon^{-1})\defeq Z(1\otimes\Upsilon^{-1})Z^{*}=q^{-2\Nhatdrinf}\otimes\Upsilon^{-1},
\end{equation}
where~\(\Nhatdrinf e_{i,j}\defeq j e_{i,j}\).
Now we recall the manageable braided multiplicative unitary over~\(\T\) relative to~\((\Corep{U},\Corep{V})\) constructed in~\cite{MRW2017}*{Theorem 4.1}:
\begin{equation}
 \label{eq:BMul-Qplane}
 \BrMultunit\defeq \brmultunit_{q}(\Upsilon q^{-2\Nhatdrinf}\otimes \Upsilon^{-1})\in\U(\Hils[L]\otimes\Hils[L]),
\end{equation}
where~\(\brmultunit_{q}\colon \conj{\C}_{(q)}\to \T\) is the quantum exponential function~\cite{W1992a}.
\begin{lemma}
 \label{lem:affl_twisted}
 The following identity holds
  \begin{equation}
    \label{eq:closure_of_R_and_S}
    \brmultunit_q(\Upsilon q^{-2\Nhatdrinf}\otimes\Upsilon^{-1})(\Upsilon^{-1}\otimes 1_{\Hils[L]})
    \brmultunit_q(\Upsilon q^{-2\Nhatdrinf}\otimes\Upsilon^{-1})^*
    =j_{1}(\Upsilon^{-1})\dotplus j_{2}(\Upsilon^{-1}) .
  \end{equation}

\end{lemma}
\begin{proof}
  Suppose, \(\widetilde{\Upsilon}\) is any closed operator acting on some Hilbert space 
  \(\Hils[L]'\) such that~\(\ker(\widetilde{\Upsilon})=\{0\}\), \(\textup{Sp}(\widetilde{\Upsilon})\subset \conj{\C}_{(q)}\) and 
  \(\Ph{\widetilde{\Upsilon}}\Mod{\widetilde{\Upsilon}}\Ph{\widetilde{\Upsilon}}^{*}
  =q^{-1}\Mod{\widetilde{\Upsilon}}\), where 
  \(\widetilde{\Upsilon}=\Ph{\widetilde{\Upsilon}}\Mod{\widetilde{\Upsilon}}\) is the 
  polar decomposition of~\(\widetilde{\Upsilon}\).  Define~\(r\defeq \widetilde{\Upsilon}\otimes\Upsilon^{-1}\otimes 1\) 
  and \(s\defeq \widetilde{\Upsilon}\otimes q^{-2\Nhatdrinf}\otimes\Upsilon^{-1}\). 
  A simple computation shows that the operators~\(r\) and~\(s\) are normal, \(\textup{Sp}(r), \textup{Sp}(s)\subseteq \conj{\C}_{(q)}\), 
  and satisfy the commutation relations in~\cite{W1992a}*{(0.1)}. By~\cite{W1992a}*{Theorem 2.2} we get
\begin{multline*}
\brmultunit_q(1\otimes \Upsilon q^{-2\Nhatdrinf}\otimes\Upsilon^{-1})
(\widetilde{\Upsilon}\otimes\Upsilon^{-1}\otimes 1) 
\brmultunit_q(1\otimes \Upsilon q^{-2\Nhatdrinf}\otimes\Upsilon^{-1})^* \\
  =\widetilde{\Upsilon}\otimes\Upsilon^{-1}\otimes 1\dotplus \widetilde{\Upsilon}\otimes q^{-2\Nhatdrinf}\otimes\Upsilon^{-1}.
\end{multline*}
Since~\(\widetilde{\Upsilon}\) is arbitrary, we have~\eqref{eq:closure_of_R_and_S}. 
\end{proof}
We shall prove that~\(\Bialg{B}\) with~\(\Comult[B](\Upsilon^{-1})\defeq j_{1}(\Upsilon^{-1})\dotplus j_{2}(\Upsilon^{-1})\) is the braided \(\Cst\)\nb-quantum group over~\(\T\) generated by~\(\BrMultunit\). For that purpose, we need to modify the techniques 
used by Woronowicz and Zakrzewski in~\cite{WZ2002}*{Theorem 4.1} and 
as the operator \(\Upsilon^{-1}\) is not normal. The following result is due to S. L. Woronowicz and it generalises~\cite{WZ2002}*{Proposition A.1}.
\begin{proposition}
 \label{Prop:tens_aff}
 Let~\(T_{i}\) be nonzero closed densely defined operator acting on~\(H_{i}\) and let 
 \(D_{i}\) be a nondegenerate \(\Cst\)\nb-subalgebra of~\(\Bound(\Hils_{i})\) for 
 \(i=1,2\). Then \((T_{1}\otimes T_{2})\eta (D_{1}\otimes D_{2})\) if 
 and only if~\(T_{1}\eta D_{1}\) and~\(T_{2}\eta D_{2}\).
\end{proposition}
\begin{proof}
 The proof of reverse implication follows from~\cite{WN1992}*{Theorem 6.1}. 
 For the other direction assume that~\((T_{1}\otimes T_{2})\eta (D_{1}\otimes D_{2})\). Then 
 \((T^{*}_{1}T_{1}\otimes T^{*}_{2}T_{2})\eta (D_{1}\otimes D_{2})\) and using 
 \cite{WZ2002}*{Proposition A.1} we obtain~\(T_{i}^{*}T_{i}\eta 
 D_{i}\) for \(i=1,2\). Therefore, \(T_{1}^{*}T_{1}\otimes 1\) and 
 \(1\otimes T_{2}^{*}T_{2}\) are affiliated to~\(D_{1}\otimes D_{2}\). Now 
 \(z_{T_{1}}\otimes z_{T_{2}}=z_{T_{1}\otimes T_{2}}f(T_{1}^{*}T_{1}\otimes 1, 1\otimes T_{2}^{*}T_{2})\) where 
 \(f\colon [0,+\infty)\times [0,+\infty)\to\R\) defined by \(f(x,y)=(1+xy)^{\frac{1}{2}}
 (1+x)^{-\frac{1}{2}}(1+y)^{-\frac{1}{2}}\). Therefore, 
 \(z_{T_{1}}\otimes z_{T_{2}}\in\Mult(D_{1}\otimes D_{2})\) and taking appropriate slices give 
 \(z_{T_{i}}\in\Mult(D_{i})\) for~\(i=1,2\). Then we know that~\((T_{1}^{*}T_{1}\otimes 1)\eta 
 (D_{1}\otimes D_{2})\) and~\(T_{1}^{*}T_{1}+1=(1-z_{T_{1}}^{*}z_{T_{1}})^{-1}\eta 
 D_{1}\). This shows that the domain of~\(T_{1}^{*}T_{1}\otimes 1\) coincides with the range of 
 \((1-z_{T_{1}}^{*}z_{T_{1}})\otimes 1\) and this implies
 \(((1-z_{T_{1}}^{*}z_{T_{1}})\otimes 1)(D_{1}\otimes D_{2})\) is dense in~\(D_{1}\otimes D_{2}\). 
 Hence, \((1-z_{T_{1}}^{*}z_{T_{1}})D_{1}\) is dense in~\(D_{1}\). 
 Similarly we can prove that~\(T_{2}\) is also affiliated to~\(D_{2}\). 
\end{proof}
In the next result, we construct the complex quantum plane as a braided \(\Cst\)\nb-quantum group~\(\Bialg{B}\) over~\(\T\).
\begin{theorem}
  \label{the:braided_qauntum_grp_qnt_plane}
   \(\Bialg{B}\) is a braided~\(\Cst\)\nb-quantum group over~\(\T\) generated by~\(\BrMultunit\). Equivalently, 
   \(B=\{(\omega\otimes\Id_{\Hils[L]})\BrMultunit \mid\text{ }\omega\in\Bound(\Hils[L])_{*}\}^{\CLS}\) and
   \(\Comult[B](\Upsilon^{-1})\defeq j_1(\Upsilon^{-1})\dotplus j_2(\Upsilon^{-1})\) is the unique 
   \(\Z\times\T\)\nb-equivariant element~\(\Comult[B]\in\Mor(B,B\boxtimes B)\) satisfying
  \eqref{eq:brd_chr_leg2}-\eqref{eq:br_cancel}.
  \end{theorem}
\begin{proof}
Let~\(B'\defeq\{(\omega\otimes\Id_{\Hils[L]})\brmultunit_{q}(\Upsilon q^{-2\Nhatdrinf}\otimes\Upsilon^{-1}) 
\mid\text{ }\omega\in\Bound(\Hils[L])_{*}\}^{\CLS}\). Then~\(B'\) is a 
\(\Cst\)\nb-algebra given by Theorem~\ref{the:braid_Qgrp}.
Since, \(\Upsilon q^{-2\Nhatdrinf}\) is a closed operator acting on~\(\Hils[L]\), it is 
 affiliated to~\(\Comp(\Hils[L])\). This implies that \(\Upsilon q^{-2\Nhatdrinf}\otimes\Upsilon^{-1}\) 
 is affiliated to~\(\Comp(\Hils[L])\otimes B\). Consequently, 
 \(\BrMultunit\in\U(\Comp(\Hils[L])\otimes B)\)
 because of \cite{W2001}*{Theorem 5.1}. Thus, from the 
 definition of \(B'\), we have \(B'\subseteq\Mult(B)\).
 
 Now~\(\BrMultunit(\Comp(\Hils[L])\otimes B)=\Comp(\Hils[L])\otimes B\) implies
\begin{align}
 \label{eq:slice_alg_qntpl}
 B'B
 &=\{(\omega\otimes\Id_{\Hils[L]})\BrMultunit (1\otimes b)\mid \text{ } 
 \omega\in\Bound(\Hils[L])_{*},\text{ }b\in B\}^\CLS \\
 &=\{(\omega\otimes\Id_{\Hils[L]})\BrMultunit (m\otimes b)\mid \text{ } 
 \omega\in\Bound(\Hils[L])_{*},\text{ }m\in\Comp(\Hils[L]),\text{ }b\in B\}^\CLS \nonumber\\
 &=\{(\omega\otimes\Id_{\Hils[L]})\BrMultunit \mid \text{ } 
 \omega\in\Bound(\Hils[L])_{*},\text{ }m\in\Comp(\Hils[L]),\text{ }b\in B\}^\CLS \nonumber
 =B.
\end{align}
To prove~\(B=B'\), it is sufficient to show that~\(B'B=B'\). We shall obtain this 
by showing the canonical embedding~\(B\hookrightarrow\Bound(\Hils[L])\) 
is an element of~\(\Mor(B,B')\). 

Define~\(T(\lambda)\defeq\brmultunit_{q}(\lambda\Upsilon q^{-2\Nhatdrinf}\otimes\Upsilon^{-1})\) 
and \(T'(\lambda)\defeq \brmultunit_{q}(\lambda\Upsilon q^{-2\Nhatdrinf}\otimes q^{-2\Nhatdrinf}\otimes \Upsilon^{-1})\)
for all~\(\lambda\in\conj{\C}_{(q)}\). By~\cite{W2001}*{Theorem 5.1}, 
\(\conj{\C}_{(q)}\ni\lambda\mapsto T(\lambda)\in\Mult(\Comp(\Hils[L])\otimes\Comp(\Hils[L]))\) and 
\(\lambda\to T'(\lambda)\in \Mult(\Comp(\Hils[L])\otimes\Comp(\Hils[L])\otimes\Comp(\Hils[L]))\) 
are continuous with respect to the strict 
topology. Therefore, \((T(\lambda)\otimes 1_{B'})_{\lambda\in\conj{\C}_{(q)}}\) is a 
continuous family of elements of~\(\Mult(\Comp(\Hils[L])\otimes\Comp(\Hils[L])\otimes B')\). 

For a fixed~\(\lambda\in  \conj{\C}_{(q)}\) we observe the operators 
\[
 R\defeq \lambda\Upsilon q^{-2\Nhatdrinf}\otimes\Upsilon^{-1} \otimes 1, 
 \qquad 
 S\defeq \lambda\Upsilon q^{-2\Nhatdrinf}\otimes q^{-2\Nhatdrinf}\otimes \Upsilon^{-1},
\]
are normal, \(\textup{Sp}(R), \textup{Sp}(S)\subseteq \conj{\C}_{(q)}\), and satisfy the commutation relations in~\cite{W1992a}*{(0.1)}. By 
\cite{W1992a}*{Theorems 2.2 \& 3.1} we get
\[
    \brmultunit_{q}(R^{-1}S)\brmultunit_{q}(R)\brmultunit_{q}(R^{-1}S)^{*}
  =\brmultunit_{q}(\brmultunit_{q}(R^{-1}S)R\brmultunit_{q}(R^{-1}S)^{*})
  =\brmultunit_{q}(R)\brmultunit_{q}(S)
\]
and this is equivalent to 
\[
 T(\lambda)_{12}^{*}\BrMultunit_{23}T(\lambda)_{12}\BrMultunit_{23}^{*} 
 =\brmultunit_{q}(\lambda\Upsilon q^{-2\Nhatdrinf}\otimes q^{-2\Nhatdrinf}\otimes \Upsilon^{-1}).
\]

Now \(\BrMultunit\in\Mult(\Comp(\Hils[L])\otimes B')\) and~\(T(\lambda)\in\Mult(\Comp(\Hils[L])\otimes\Comp(\Hils[L]))\) 
implies \(T'(\lambda)\in\Mult(\Comp(\Hils[L])\otimes\Comp(\Hils[L])\otimes B')\) for all~\(\lambda\in  \conj{\C}_{(q)}\). This shows that 
\(\lambda\mapsto T'(\lambda)\in \Mult(\Comp(\Hils[L])\otimes\Comp(\Hils[L])\otimes B')\) is continuous with respect to the 
strict topology. Therefore, \(\Upsilon q^{-2\Nhatdrinf}\otimes q^{-2\Nhatdrinf}\otimes \Upsilon^{-1}\) is affiliated to 
\(\Comp(\Hils[L])\otimes\Comp(\Hils[L])\otimes B'\); hence, 
\(\Upsilon^{-1}\) is affiliated to~\(B'\) by Proposition~\ref{Prop:tens_aff}. 
Since~\(\Upsilon^{-1}\) generates \(B\) and is affiliated to~\(B'\) the embedding 
\(B\hookrightarrow\Bound(\Hils[L])\) is an element of \(\Mor(B,B')\), 
see~\cite{W1995}*{Definition 3.1}.  Consequently, Lemma~\ref{lem:affl_twisted} shows~\(\Comult[B](\Upsilon^{-1})=j_1(\Upsilon^{-1})\dotplus j_2(\Upsilon^{-1})\). Finally, since \(\Upsilon^{-1}\eta B\) and \(\Comult[B]\in\Mor(B,B\boxtimes B)\) so is \(\Comult[B](\Upsilon^{-1}) \eta B\boxtimes B\). 
  \end{proof}
\subsubsection{Dual of the complex quantum plane}
Suppose~\(\DuBialg{B}\) is the dual braided \(\Cst\)\nb-quantum group of~\(\Bialg{B}\) over~\(\Z\) generated by the dual of~\(\BrMultunit\) given by~\cite{MRW2017}*{Definition 3.3 \& Proposition 3.4}. Since the roles of~\(\G\) and~\(\DuG\) are exchanged it yield the changes in the braiding operator~\(\Dualbraiding{\Hils[L]}{\Hils[L]}\defeq (\Braiding{\Hils[L]}{\Hils[L]})^{*}=\widehat{Z}\Flip\) with~\(\widehat{Z}\defeq \Flip Z^{*}\Flip\) and~\(\hat{B}\) is an object  \(\YDcat(\DuG)\). A variant of the Proposition~\ref{prop:Cstar-q-plane} shows that~\(\hat{B}\) is generated by~\(\Upsilon\). Let \(\widehat{\boxtimes}\) be the monoidal product of~\(\YDcat(\DuG)\) and \(i_{1}, i_{2}\) be the canonical morphisms \(\hat{B}\to \hat{B}\widehat{\boxtimes}\hat{B}\) defined by~\eqref{eq:twisted_tensor} (with respect to the braiding 
\(\Dualbraiding{\Hils[L]}{\Hils[L]}\)). On~\(\Upsilon\)  they are defined by 
\begin{equation}
  \label{eq:Embed-hatBboxB}
i_{1}(\Upsilon)\defeq\Upsilon\otimes 1, 
\qquad 
i_{2}(\Upsilon)\defeq \widehat{Z}(1\otimes\Upsilon)\widehat{Z}^{*}=\udrinf\otimes \Upsilon.
\end{equation}
A similar analysis describes the dual~\(\DuBialg{B}\) of~\(\Bialg{B}\) as a braided quantum group over~\(\Z\).
\begin{corollary}
   \(\hat{B}=\{(\omega\otimes\Id_{\Hils[L]})\DuBrMultunit\mid \omega\in\Bound(\Hils[L])_{*}\}^\CLS\) and 
   is a \(\DuG\)\nb-Yetter-Drinfeld \(\Cst\)\nb-algebra with respect to the~\(\DuG\) and~\(\G\) actions 
   \(\hat{\delta}\) and~\(\delta\) defined by~\(\Upsilon\to \Upsilon\otimes u\) and~\(\Upsilon\to \Upsilon\otimes q^{2\hat{N}}\), respectively. 
   The sum \(i_{1}(\Upsilon)\dotplus i_2(\Upsilon)\) is affiliated 
  to~\(\hat{B}\widehat{\boxtimes} \hat{B}\). The map \(\DuComult[B](\Upsilon)\defeq i_1(\Upsilon)\dotplus i_2(\Upsilon)\) 
  is the unique \(\T\times\Z\)\nb-equivariant element~\(\DuComult[B]\in\Mor(\hat{B},\hat{B}\widehat{\boxtimes}\hat{B})\) satisfying
  \eqref{eq:brd_chr_leg2}-\eqref{eq:br_cancel} for the dual of~\(\BrMultunit\). 
  Thus, \(\DuBialg{B}\) is a braided~\(\Cst\)\nb-quantum group over~\(\Z\).
\end{corollary}
We may also realise~\(\hat{B}\) as a~\(\T\)\nb-Yetter\nb-Drinfeld \(\Cst\)\nb-algebra. 
On the other hand, a simple observation shows that the polar decomposition \(\Upsilon^{*}=\Ph{\Upsilon^{*}}\Mod{\Upsilon^{*}}\) gives a unitary operator~\(\Ph{\Upsilon^*}\), a strictly positive operator~\(\Mod{\Upsilon^*}\) with spectrum~\(q^{\Z}\cup \{0\}\) and satisfy the commutation relation~\eqref{eq:prop_of_Upsilon}. Since, \(\hat{B}\) is also generated by~\(\Upsilon^{*}\), the map~\(f\colon \Upsilon^{-1}\to\Upsilon^{*}\) extends to an isomorphism between \(B\) and~\(\hat{B}\) in the category of~\(\Z\times\T\)\nb-\(\Cst\)\nb-algebras. Thus \(B\) is isomorphic to~\(\hat{B}\) also in the category of~\(\T\times\Z\)\nb-\(\Cst\)\nb-algebras. Consequently, \(f\) an isomorphism of braided \(\Cst\)\nb-quantum groups between \((B,\DuComult[B])\) and \(\Bialg{B}\) over~\(\Z\).

\subsubsection{The bosonization}
Now we describe the quantum group with projection~\(\Bialg{C}\) in Proposition~\ref{prop:big_qntgrp} 
associated to the quantum plane \(\Bialg{B}\). Here~\(\G\) is the compact group~\(\T\) 
viewed as a quantum group then~\(C=\Cont(\T)\boxtimes B\). In fact~\(C\cong B\rtimes_{\hat{\gamma}}\Z\), 
where~\(\hat{\gamma}\) is defined by~\(\hat{\gamma}_{m}(\Upsilon^{-1})=q^{-2m}\Upsilon^{-1}\).
The embeddings of~\(\Cont(\T)\) and~\(B\) are given by~\(u\mapsto u\otimes 1\) and~\(\Upsilon^{-1}\mapsto 
\Ducorep{V}^{*}(1\otimes\Upsilon^{-1})\Ducorep{V}=q^{-2\hat{N}}\otimes\Upsilon^{-1}\). Using the 
definitions of the unitaries~\(\Corep{U}\), \(\Multunit\), \(\DuCorep{V}\) and \(\BrMultunit\) we compute that 
\((\Multunit[C])(u\otimes 1\otimes 1\otimes 1)(\Multunit[C])^{*}=u\otimes 1\otimes u\otimes 1\) and 
\begin{align*}
 & \Multunit[C](q^{-2\hat{N}}\otimes \Upsilon^{-1}\otimes 1\otimes 1)(\Multunit[C])^{*} \\
 &= \Multunit_{13}\Corep{U}_{23}\DuCorep{V}_{34}^{*}\BrMultunit_{24} 
       (q^{-2\hat{N}}\otimes \Upsilon^{-1}\otimes 1\otimes 1)
       \BrMultunit_{24}\DuCorep{V}_{34}\Corep{U}_{23}^{*}\Multunit[*]_{13}\\
 &= \Multunit_{13}\Corep{U}_{23}\DuCorep{V}_{34}^{*}
       \bigl(q^{-2\hat{N}}\otimes (\Upsilon^{-1}\otimes 1\otimes 1\dotplus q^{-2\Nhatdrinf}\otimes 1\otimes \Upsilon^{-1})\bigr)
       \DuCorep{V}_{34}\Corep{U}_{23}^{*}\Multunit[*]_{13}\\
 &= \Multunit_{13}\Corep{U}_{23}
       \bigl(q^{-2\hat{N}}\otimes (\Upsilon^{-1}\otimes 1\otimes 1\dotplus q^{-2\Nhatdrinf}\otimes q^{-2\hat{N}}\otimes \Upsilon^{-1})\bigr)
       \Corep{U}_{23}^{*}\Multunit[*]_{13}\\       
 &= \Multunit_{13}
       \bigl(q^{-2\hat{N}}\otimes (\Upsilon^{-1}\otimes u^{*}\otimes 1\dotplus 1\otimes q^{-2\hat{N}}\otimes \Upsilon^{-1})\bigr)
       \Multunit[*]_{13}\\       
  &=  q^{-2\hat{N}}\otimes \Upsilon^{-1}\otimes u^{*}\otimes 1 \dotplus 1\otimes 1\otimes q^{-2\hat{N}}\otimes \Upsilon^{-1}.
\end{align*}
Define~\(\Psi\defeq q^{-2\hat{N}}\otimes \Upsilon^{-1}\) and~\(V\defeq u^{*}\otimes 1\). Then 
\(C\) is the universal \(\Cst\)\nb-algebra generated by~\(\Psi\) and~\(V\) satisfying the following (formal) 
relations
\begin{equation}
 \label{eq:bosonized}
  V^{*}V=VV^{*}=1, 
  \qquad 
  \Psi^{*}\Psi=q^{-2}\Psi\Psi^{*}, 
  \qquad 
  \textup{Sp}(\Mod{\Psi})=q^{\Z}\cup\{0\}, 
  \qquad 
  V\Psi V^{*}= q^{-2}\Psi, 
\end{equation}
and the comultiplication map~\(\Comult[C]\in\Mor(C, C\otimes C)\) is given by 
\begin{equation*}
 \Comult[C](V)=V\otimes V, 
 \qquad 
 \Comult[C](\Psi)=\Psi\otimes V\dotplus 1\otimes \Psi.
\end{equation*}
In fact,~\(\Bialg{C}\) are closely related to \(\textup{E}_{q}(2)\) groups~\cites{W1991b}. For 
a fixed \(0<q<1\) the quantum \(\textup{E}(2)\) group~\(\Bialg{C'}\) is described by a unitary operator \(v\) and a normal 
operator~\(n\) with \(\textup{Sp}(\Mod{n})=q^{\Z}\cup\{0\}\). Underlying \(\Cst\)\nb-algebra \(C'\) is generated by~\(v\) and~\(n\) 
subject to the commutation relation~\(v^{*}nv=qn\) and~\(\Comult[C']\in\Mor(C',C'\otimes C')\) is defined by \(\Comult[C'](v)=v\otimes v\) 
and~\(\Comult[C'](n)=v\otimes n\dotplus n\otimes v^{*}\). A simple observations show that~\(V=v^{2}\) and~\(\Psi =v^{*}n\) satisfy~\eqref{eq:bosonized} 
and~~\(\Comult[C]|_{C'}=\Comult[C]\). Therefore, there exists a unique Hopf~\Star{}homomorphism~\(f\colon C\to C'\) such that~\(f(V)=v^{2}\) 
and~\(f(\Psi)=v^{*}n\). The image of~\(\Bialg{C}\) inside~\(\Bialg{C'}\) was constructed by Woronowicz (in an unpublished work) under the name simplified 
quantum \(E(2)\) groups.  


\begin{bibdiv}
  \begin{biblist}
\bib{BS1993}{article}{
  author={Baaj, Saad},
  author={Skandalis, Georges},
  title={Unitaires multiplicatifs et dualit\'{e} pour les produits crois\'{e}s de {$C^*$}-alg{\`e}bres},
  date={1993},
  issn={0012-9593},
  journal={Ann. Sci. \'{E}cole Norm. Sup. (4)},
  volume={26},
  number={4},
  pages={425\ndash 488},
  eprint={http://www.numdam.org/item?id=ASENS_1993_4_26_4_425_0},
}

\bib{BJR2022}{article}{
      author={Bhattacharjee, Suvrajit},
      author={Joardar, Soumalya},
      author={Roy, Sutanu},
       title={Braided quantum symmetries of graph $\mathrm{C}^*$-algebras},
       publisher={arXiv},
       date={2022},
       doi={10.48550/arXiv.2201.09885},
}


\bib{DKSS2012}{article}{
  author={Daws, Matthew},
  author={Kasprzak, Pawe\l },
  author={Skalski, Adam},
  author={So\l tan, Piotr~M.},
  title={Closed quantum subgroups of locally compact quantum groups},
  date={2012},
  issn={0001-8708},
  journal={Adv. Math.},
  volume={231},
  number={6},
  pages={3473\ndash 3501},
  doi={10.1016/j.aim.2012.09.002},
}

\bib{KMRW2016}{article}{
  author={Kasprzak, Pawe\l },
  author={Meyer, Ralf},
  author={Roy, Sutanu},
  author={Woronowicz, Stanis\l aw~L.},
  title={Braided quantum {$\rm SU(2)$} groups},
  date={2016},
  issn={1661-6952},
  journal={J. Noncommut. Geom.},
  volume={10},
  number={4},
  pages={1611\ndash 1625},
  doi={10.4171/JNCG/268},
}

\bib{KS2015}{article}{
  author={Kasprzak, Pawe\l },
  author={So\l tan, Piotr~M.},
  title={Quantum groups with projection on von {N}eumann algebra level},
  date={2015},
  issn={0022-247X},
  journal={J. Math. Anal. Appl.},
  volume={427},
  number={1},
  pages={289\ndash 306},
  doi={10.1016/j.jmaa.2015.02.047},
}

\bib{KV2000}{article}{
  author={Kustermans, Johan},
  author={Vaes, Stefaan},
  title={Locally compact quantum groups},
  date={2000},
  issn={0012-9593},
  journal={Ann. Sci. \'{E}cole Norm. Sup. (4)},
  volume={33},
  number={6},
  pages={837\ndash 934},
  doi={10.1016/S0012-9593(00)01055-7},
}

\bib{M1994}{article}{
  author={Majid, Shahn},
  title={Cross products by braided groups and bosonization},
  date={1994},
  issn={0021-8693},
  journal={J. Algebra},
  volume={163},
  number={1},
  pages={165\ndash 190},
  doi={10.1006/jabr.1994.1011},
}

\bib{M1995}{book}{
  author={Majid, Shahn},
  title={Foundations of quantum group theory},
  publisher={Cambridge University Press, Cambridge},
  date={1995},
  isbn={0-521-46032-8},
  doi={10.1017/CBO9780511613104},
}

\bib{M1999}{article}{
   author={Majid, S.},
   title={Double-bosonization of braided groups and the construction of $U_q(\mathfrak{g})$},
   date={1999},   
   issn={0305-0041},   
   journal={Math. Proc. Cambridge Philos. Soc.},
   volume={125},
   number={1},
   pages={151--192},
   doi={10.1017/S0305004198002576},
}

\bib{M2000}{article}{
   author={Majid, Shahn},
   title={Braided-Lie bialgebras},
   date={2000},   
   issn={0030-8730},
   journal={Pacific J. Math.},
   volume={192},
   number={2},
   pages={329--356},
   doi={10.2140/pjm.2000.192.329},
}

\bib{MR2019a}{article}{
  author={Meyer, Ralf},
  author={Roy, Sutanu},
  title={Braided free orthogonal quantum groups},
  date={2022},
  issn={1073-7928},
  journal={Int. Math. Res. Not. IMRN},  
  number={12},
  pages={8890--8915},
   doi={10.1093/imrn/rnaa379},
   }


\bib{MR2019}{proceedings}{
  author={Meyer, Ralf},
  author={Roy, Sutanu},
  title={Braided multiplicative unitaries as regular objects},
  publisher={Mathematical Society of Japan},
  date={2019},
  volume={80},
  isbn={9784864970792},
  doi={10.2969/aspm/08010153},
}

\bib{MRW2012}{article}{
  author={Meyer, Ralf},
  author={Roy, Sutanu},
  author={Woronowicz, Stanis\l aw~L.},
  title={Homomorphisms of quantum groups},
  date={2012},
  issn={1867-5778},
  journal={M\"{u}nster J. Math.},
  volume={5},
  pages={1\ndash 24},
  eprint={http://nbn-resolving.de/urn:nbn:de:hbz:6-88399662599},
}

\bib{MRW2014}{article}{
  author={Meyer, Ralf},
  author={Roy, Sutanu},
  author={Woronowicz, Stanis\l aw~L.},
  title={Quantum group-twisted tensor products of {C{$^*$}}-algebras},
  date={2014},
  issn={0129-167X},
  journal={Internat. J. Math.},
  volume={25},
  number={2},
  pages={1450019, 37},
  doi={10.1142/S0129167X14500190},
}

\bib{MRW2016}{article}{
  author={Meyer, Ralf},
  author={Roy, Sutanu},
  author={Woronowicz, Stanis\l aw~L.},
  title={Quantum group-twisted tensor products of {${\rm C}^*$}-algebras. {II}},
  date={2016},
  issn={1661-6952},
  journal={J. Noncommut. Geom.},
  volume={10},
  number={3},
  pages={859\ndash 888},
  doi={10.4171/JNCG/250},
}

\bib{MRW2017}{article}{
  author={Meyer, Ralf},
  author={Roy, Sutanu},
  author={Woronowicz, Stanis\l aw~L.},
  title={Semidirect products of {$\rm C^*$}-quantum groups: multiplicative unitaries approach},
  date={2017},
  issn={0010-3616},
  journal={Comm. Math. Phys.},
  volume={351},
  number={1},
  pages={249\ndash 282},
  doi={10.1007/s00220-016-2727-3},
}

\bib{NV2010}{article}{
  author={Nest, Ryszard},
  author={Voigt, Christian},
  title={Equivariant {P}oincar\'{e} duality for quantum group actions},
  date={2010},
  issn={0022-1236},
  journal={J. Funct. Anal.},
  volume={258},
  number={5},
  pages={1466\ndash 1503},
  doi={10.1016/j.jfa.2009.10.015},
}

\bib{PV1997}{article}{
   author={Podle\'{s}, P.},
   author={Woronowicz, Stanis\l aw L.},
   title={On the structure of inhomogeneous quantum groups},
   journal={Comm. Math. Phys.},
   volume={185},
   date={1997},
   number={2},
   pages={325\ndash 358},
   doi={10.1007/s002200050093},
}

\bib{R1985}{article}{
  author={Radford, David~E.},
  title={The structure of {H}opf algebras with a projection},
  date={1985},
  issn={0021-8693},
  journal={J. Algebra},
  volume={92},
  number={2},
  pages={322\ndash 347},
  doi={10.1016/0021-8693(85)90124-3},
}

\bib{RR2020}{article}{
  author={Rahaman, Atibur},
  author={Roy, Sutanu},
  title={Quantum $E(2)$ groups for complex deformation parameters},
  journal={Rev. Math. Phys.},
  volume={33},
  date={2021},
  number={6},
  pages={Paper No. 2150021, 28},
  issn={0129-055X},
  doi={10.1142/S0129055X21500215},
}
  
\bib{R2013}{thesis}{
  author={Roy, Sutanu},
  title={{$\rm C^*$}-quantum groups with projection},
  type={phdthesis},
  institution={Georg-August Universit\"at G\"ottingen},
  date={2013},
  eprint={http://hdl.handle.net/11858/00-1735-0000-0022-5EF9-0},
}

\bib{R2015a}{article}{
  author={Roy, Sutanu},
  title={The {D}rinfeld double for {$C^*$}-algebraic quantum groups},
  date={2015},
  issn={0379-4024},
  journal={J. Operator Theory},
  volume={74},
  number={2},
  pages={485\ndash 515},
  doi={10.7900/jot.2014sep04.2053},
}
\bib{R2021}{article}{
      author={Roy, Sutanu},
      title={Homogeneous quantum symmetries of finite spaces over the circle group},
      date={2021},     
      publisher={arXiv},       
      doi={10.48550/arXiv.2105.01556},
}

\bib{RW2018}{article}{
  author={Roy, Sutanu},
  author={Woronowicz, Stanis\l aw~L.},
  title={Landstad-{V}aes theory for locally compact quantum groups},
  date={2018},
  issn={0129-167X},
  journal={Internat. J. Math.},
  volume={29},
  number={4},
  pages={1850028, 36},
  doi={10.1142/S0129167X18500283},
}

\bib{S2005}{article}{
      author={So\l tan, Piotr M.},
       title={New quantum ``{$az+b$}'' groups},
        date={2005},
        ISSN={0129-055X},
     journal={Rev. Math. Phys.},
      volume={17},
      number={3},
       pages={313\ndash 364},
         doi={10.1142/S0129055X05002339},
}

\bib{SW2007}{article}{
  author={So\l tan, Piotr~M.},
  author={Woronowicz, Stanis\l aw~L.},
  title={From multiplicative unitaries to quantum groups. {II}},
  date={2007},
  issn={0022-1236},
  journal={J. Funct. Anal.},
  volume={252},
  number={1},
  pages={42\ndash 67},
  doi={10.1016/j.jfa.2007.07.006},
}

\bib{S2010}{article}{
  author={So\l tan, Piotr M.},
  title={Examples of non-compact quantum group actions},
  date={2010},
  issn={0022-247X},
  journal={J. Math. Anal. Appl.},
  volume={372},
  number={1},
  pages={224\ndash 236},
  doi={10.1016/j.jmaa.2010.06.045},
}

\bib{V2005}{article}{
  author={Vaes, Stefaan},
  title={A new approach to induction and imprimitivity results},
  date={2005},
  issn={0022-1236},
  journal={J. Funct. Anal.},
  volume={229},
  number={2},
  pages={317\ndash 374},
  doi={10.1016/j.jfa.2004.11.016},
}

\bib{W1991a}{article}{
  author={Woronowicz, Stanis\l aw~L.},
  title={Quantum {$E(2)$} group and its {P}ontryagin dual},
  date={1991},
  issn={0377-9017},
  journal={Lett. Math. Phys.},
  volume={23},
  number={4},
  pages={251\ndash 263},
  doi={10.1007/BF00398822},
}

\bib{W1991b}{article}{
  author={Woronowicz, Stanis\l aw~L.},
  title={Unbounded elements affiliated with {$C^*$}-algebras and noncompact quantum groups},
  date={1991},
  issn={0010-3616},
  journal={Comm. Math. Phys.},
  volume={136},
  number={2},
  pages={399\ndash 432},
}

\bib{W1992a}{article}{
  author={Woronowicz, Stanis\l aw~L.},
  title={Operator equalities related to the quantum {$E(2)$} group},
  date={1992},
  issn={0010-3616},
  journal={Comm. Math. Phys.},
  volume={144},
  number={2},
  pages={417\ndash 428},
}

\bib{W1995}{article}{
  author={Woronowicz, Stanis\l aw~L.},
  title={{$C^*$}-algebras generated by unbounded elements},
  date={1995},
  issn={0129-055X},
  journal={Rev. Math. Phys.},
  volume={7},
  number={3},
  pages={481\ndash 521},
  doi={10.1142/S0129055X95000207},
}

\bib{W1996}{article}{
  author={Woronowicz, Stanis\l aw~L.},
  title={From multiplicative unitaries to quantum groups},
  date={1996},
  issn={0129-167X},
  journal={Internat. J. Math.},
  volume={7},
  number={1},
  pages={127\ndash 149},
  doi={10.1142/S0129167X96000086},
}

\bib{W2001}{article}{
  author={Woronowicz, Stanis\l aw~L.},
  title={Quantum ``{$az+b$}'' group on complex plane},
  date={2001},
  issn={0129-167X},
  journal={Internat. J. Math.},
  volume={12},
  number={4},
  pages={461\ndash 503},
  doi={10.1142/S0129167X01000836},
}

\bib{WN1992}{article}{
  author={Woronowicz, Stanis\l aw~L.},
  author={Napi\'{o}rkowski, K.},
  title={Operator theory in the {$C^\ast $}-algebra framework},
  date={1992},
  issn={0034-4877},
  journal={Rep. Math. Phys.},
  volume={31},
  number={3},
  pages={353\ndash 371},
  doi={10.1016/0034-4877(92)90025-V},
}

\bib{WZ2002}{article}{
  author={Woronowicz, Stanis\l aw~L.},
  author={Zakrzewski, Stanis\l aw},
  title={Quantum `{$ax+b$}' group},
  date={2002},
  issn={0129-055X},
  journal={Rev. Math. Phys.},
  volume={14},
  number={7-8},
  pages={797\ndash 828},
  doi={10.1142/S0129055X02001405},
}
  \end{biblist}
\end{bibdiv}
\end{document}